\documentclass[%
 aip,
 amsmath,amssymb,
 reprint,%
]{revtex4-1}

\usepackage[english]{babel}
\usepackage{graphicx}%
\usepackage{dcolumn}%
\usepackage{bm}%

\usepackage[utf8]{inputenc}
\usepackage[T1]{fontenc}
\usepackage{mathptmx}
\usepackage{etoolbox}

\makeatletter
\def\@email#1#2{%
 \endgroup
 \patchcmd{\titleblock@produce}
  {\frontmatter@RRAPformat}
  {\frontmatter@RRAPformat{\produce@RRAP{*#1\href{mailto:#2}{#2}}}\frontmatter@RRAPformat}
  {}{}
}%
\makeatother

\graphicspath{{./Figures/}}

\newcommand{\revision}[1]{{\color{black} #1}}

\renewcommand{\epsilon}{\varepsilon}
\newcommand{\e}{\varepsilon}
\newcommand{\R}{\mathbb{R}}
\newcommand{\N}{\mathbb{N}}
\newcommand{\cX}{\mathcal{X}}

\usepackage{float}
\usepackage{amsthm,amsfonts,}
\usepackage{thmtools}
\usepackage{thm-restate}
\usepackage[colorlinks=true, allcolors=blue]{hyperref}

\newtheorem{definition}{Definition}[section]

\newtheorem{lemma}{Lemma}[section]

\usepackage{dsfont}

\begin{document}

\preprint{AIP/123-QED}

\title[Detecting bifurcations in dynamical systems with CROCKER plots]{Detecting bifurcations in dynamical systems with CROCKER plots}

\author{\.{I}smail G\"{u}zel}
\email{iguzel@itu.edu.tr}
\homepage{https://ismailguzel.github.io/}
\affiliation{Department of Mathematics Engineering, \.{I}stanbul Technical University, \revision{Maslak, \.{I}stanbul 34469, T\"{u}rkiye}}%

\author{Elizabeth Munch}%
\email{muncheli@msu.edu}
\homepage{https://elizabethmunch.com/}
\affiliation{ 
Department of Computational Mathematics, Science and Engineering, Department of Mathematics, Michigan State University, \revision{East Lansing, Michigan 48824, USA.}%
}%

\author{Firas A.~Khasawneh}
\email{khasawn3@msu.edu}
\homepage{https://firaskhasawneh.com/}
\affiliation{%
Mechanical Engineering, Michigan State University, \revision{East Lansing, Michigan 48824, USA.}%
}%

\date{\today}%

\begin{abstract}
Existing tools for bifurcation detection from signals of dynamical systems typically are either limited to a special class of systems, or they require carefully chosen input parameters, and significant expertise to interpret the results. Therefore, we describe an alternative method based on persistent homology---a tool from Topological Data Analysis (TDA)---that utilizes Betti numbers and CROCKER plots. 
Betti numbers are topological invariants of topological spaces, while the CROCKER plot is a coarsened but easy to visualize data representation of a one-parameter varying family of persistence barcodes. 
The specific bifurcations we investigate are transitions from periodic to chaotic behavior or vice versa in a one-parameter \revision{collection} of differential equations.
We validate our methods using numerical experiments on ten dynamical systems 
and contrast the results with existing tools that use the maximum Lyapunov exponent. 
We further prove the relationship between the Wasserstein distance to the empty diagram and the norm of the Betti vector, which shows that an even more simplified version of the information has the potential to provide insight into the bifurcation parameter.
The results show that our approach reveals more information about the shape of the periodic attractor than standard tools, and it has more favorable computational time in comparison to the R\"{o}senstein algorithm \revision{for computing the maximum Lyapunov exponent.}  

\end{abstract}

\maketitle

\begin{quotation}
Qualitative changes in the system response such as transitions from regular to chaotic behavior are referred to as bifurcations. 
When the system is modeled by differential equations with a large number of degrees of freedom, or if all that is available is a time series from a system observable, detecting these bifurcations becomes challenging.
Therefore, we focus on a data-driven approach for bifurcation detection where we measure properties of embedded time series to determine if and when the behavior of the system has changed.
The tools used come from topological data analysis (TDA), where shape is quantified using tools from algebraic topology. 
The contribution of this paper is to show how the CROCKER plot \cite{crocker-Lori2015}, a visual tool for analysis of changing shape, can be used for the purpose of analyzing and detecting bifurcations. 
Further, compressed versions of this information by taking the norm over columns of the resulting matrix can also be associated to changes in the system. 
The analysis is performed by testing the ideas empirically on ten dynamical systems and providing comparisons to standard bifurcation analysis tools.

\end{quotation}

\section{Introduction}

Systems expressing chaotic behavior can be found in a variety of domains, ranging from mathematics to biology, economics to electrical circuits, and engineering to social sciences. 
When considering chaotic systems in terms of engineering applications, the major issue is determining how to control complexity and unpredictability. 
As the number of features in established systems grows, the system becomes more complicated, making modeling the system more difficult. 
Further, since the definition of bifurcation is quite qualitative, we are left to either qualitative measurements, or turn to other quantitative measures which tend to be computationally expensive options for analysis. 
In this paper, we propose a method for investigating bifurcations in dynamical systems using a tool from topological data analysis (TDA) called the CROCKER plot \cite{crocker-Lori2015}. 

There is a growing literature dedicated to connecting dynamical systems analysis with TDA to create new data-driven analysis tools, a field collectively known as Topological Signal Processing (TSP) \cite{robinson2014topological}. 
The original seeds of TDA as a field were planted through connections with dynamical systems \cite{Robins2000,Kaczynski2004}, so it is no surprise that these ideas can be quite useful for time series analysis.
Much of the recent work utilizes \textit{persistent homology}, colloquially known as persistence, which encodes the shape and structure of given input data by providing measurements of features such as clusters and holes.
The standard and highly utilized pipeline for time series data
\cite{khasawneh2017utilizing,Garland2016,mittal2017topological,TACTS,Guzel2022stochastic,Karan2021,maletic2016persistent,Marchese2017,Kim2018c,Seversky2016,Venkataraman2016,Perea2014,Perea2019a,Berwald2014a,Tan2021} 
is as follows. 
Take a time series of interest, generate a point cloud from it using a delay coordinate embedding, create a filtration of the complete simplicial complex using the Rips construct, and compute the persistent homology of the result. 
From here, the analysis tools depend on the types of input data and conclusions desired, but nearly all involve some form of featurization of the resulting persistence diagrams.
This process, which returns some form a vector representing the persistence diagram, is required as the unusual geometry of the space of diagrams \cite{Bubenik2020a} results in limitations when passing the information as inputs to machine learning algorithms. 
This pipeline has found a great deal of success in both theory and applications, including 
machining \cite{Yesilli2019b, Khasawneh2015, Khasawneh2018},
finance \cite{Gidea2017a, Gidea2018a, Gidea2020, RiveraCastro2019,Majumdar2020},
and biomedicine \cite{Ignacio2019a,Stolz2017,Majumder2020,Chung2021,Emrani2014}. 
More recently, work has begun to modify aspects of this pipeline \cite{tymochko2020using,myers2019persistent,Myers2022,Bauer2020b,Tempelman2020,Khasawneh2018}
or to accept different forms of data input than simple time series 
\cite{Myers2022a,Tralie2016a,Tymochko2020,Tralie2017,Levanger2019,Wu2021}.

In this paper, we focus on the case of time series data associated to a bifurcation parameter.
That is, a parameter which can be changed in the system resulting in different types of output behavior. 
With a collection of time series, each associated to this parameter value, we can employ tools from TDA which are built to handle a 1-parameter family of persistence diagrams. 
While this collection of methods includes 
multiparameter persistence \cite{multiparameterrank},
and vineyards \cite{CohenSteiner2006}, 
we will focus on the CROCKER plot \cite{crocker-Lori2015} due to its simple visualization. 
The idea is that for each value of the bifurcation parameter, there is a persistence diagram computed from the time series via the normal pipeline, but this can instead be encoded via its simpler Betti curve. 
These Betti curves, which are vectorized into columns, can be collected into a matrix to easily visualize changes in the persistence diagram structure, which can then be interpreted against the original bifurcation parameter for analysis of the system. 

The result, called a CROCKER plot, was originally developed in the context of dynamic metric spaces \cite{crocker-Lori2015,crocker-Lori2019,bhaskar2019analyzing,crocker-stack}.
{There, the  additional restrictions} means that the time-varying point clouds under study have labels on vertices from one parameter value to the next, allowing for more available theoretical results on continuity. 
However, the visualization tool itself needs no such assumption, as we have no tracking information on our point clouds from one step to the next. 

In this paper, we show that the CROCKER plot can be used as both a qualitative and quantitative tool for understanding bifurcations in dynamical systems. 
We will also begin to look at an even more simplified version of the information, namely the $L_1$ norm of the Betti curves. 
As we show that this construction is closely related to the 1-Wasserstein distance commonly used for persistence diagrams and make connections between this and the maximum Lyapunov exponent, a commonly used measure for chaos. 
This paper is the full version of our short conference paper \cite{Guezel2022}.

\paragraph*{Outline.} 
We organize this paper as follows. 
In Sec.~\ref{sec:background} we give the necessary background on dynamical systems, persistent homology, and CROCKER plots. 
In Sec.~\ref{sec:Experiments}, we give specifics of the method, and show the results with full details on the Lorenz and R\"{o}ssler systems, with further results shown for a longer list of example dynamical systems. 
Finally, we discuss conclusions, limitations, and future work in Sec.~\ref{sec:Conclusions}.

\section{Background}
\label{sec:background}

In this section, we give the necessary background for the CROCKER based analysis of dynamical systems. 
In particular, we discuss needed dynamical systems background in Sec.~\ref{ssec:DynamicalSystems}; and the topological tools we use in Sec.~\ref{ssec:TDA}.

\subsection{Dynamical systems}
\label{ssec:DynamicalSystems}

Throughout this work, we assume that we have access to sampled realizations of a dynamical system $X = [x_1,\cdots,x_N]$ where $x_i \in \R^n$ and have the goal of analyzing changes in the behavior in a data-driven manner. 
In this paper, we use two main tools from the dynamics literature for analyzing the dynamics of a non-linear system. 
Namely, we first use the bifurcation diagram which is a qualitative measure
and Lyapunov exponent which is a quantitative measure.
As these tools are standard in nonlinear time series analysis, we direct the interested reader to texts\cite{Kantz2004,Hirsch2003,baker1996chaotic, sayama2015introduction} for further details.

\subsubsection{Bifurcation Diagrams}

A bifurcation in a dynamical system is characterized by tiny changes in parameter values causing {qualitative} changes in the behavior. 
These result in {qualitative} changes in the system's equilibrium point or the unstable state of stable equilibrium points. 
Hence, detecting them is critical since they can indicate when a system is transitioning from normal operation to imminent breakdown. 
Bifurcations can occur in both continuous and discrete time systems, and we call the system parameter that produces the bifurcation event the \textit{bifurcation parameter}.
One visual tool for finding bifurcations is the bifurcation diagram, which shows local extrema of a given system over a varying control parameter while keeping other parameters fixed.

An example of this can be seen in Fig.~\ref{fig:LorenzBifurcLyapunov}(a) for the commonly studied Lorenz system \cite{lorenz1963deterministic}. 
This system is given by the equations
    \begin{equation*}
        \dot{x} = \sigma(y-x),\quad
        \dot{y} = x(\rho-z)-y,\quad
        \dot{z} = xy-\beta z,
    \end{equation*}
where the constants $ \sigma, \rho, \beta $ are system parameters, $x$ is proportional to the rate of convection, and $y$ and $z$ are the horizontal and vertical temperature variation, respectively. 
In the example of Fig.~\ref{fig:LorenzBifurcLyapunov}(a), we show the bifurcation diagram with respect to varying the parameter $\rho$ with 600 equally spaced values between 90 and 105; and with the $\sigma$ and $\beta$ parameters remaining constant.   
The attractor for four marked values of $\rho$ can be seen above the bifurcation diagram in Fig.~\ref{fig:LorenzBifurcLyapunov}.
\revision{The Lorenz system was simulated using \texttt{DynamicalSystems.jl} ~\cite{juliaDynamicalSystems} package in julia~\cite{bezanson2017julia} with default function parameters.}
As we see from Fig.~\ref{fig:LorenzBifurcLyapunov}, there are two obvious periodic regions around $\rho = 92.1$ and $\rho = 100.1$.
However, a limitation of bifurcation diagrams is their inherently qualitative nature, meaning further tools are required for understanding bifurcations.

    \begin{figure*}
        \centering
        \includegraphics[width=\linewidth]{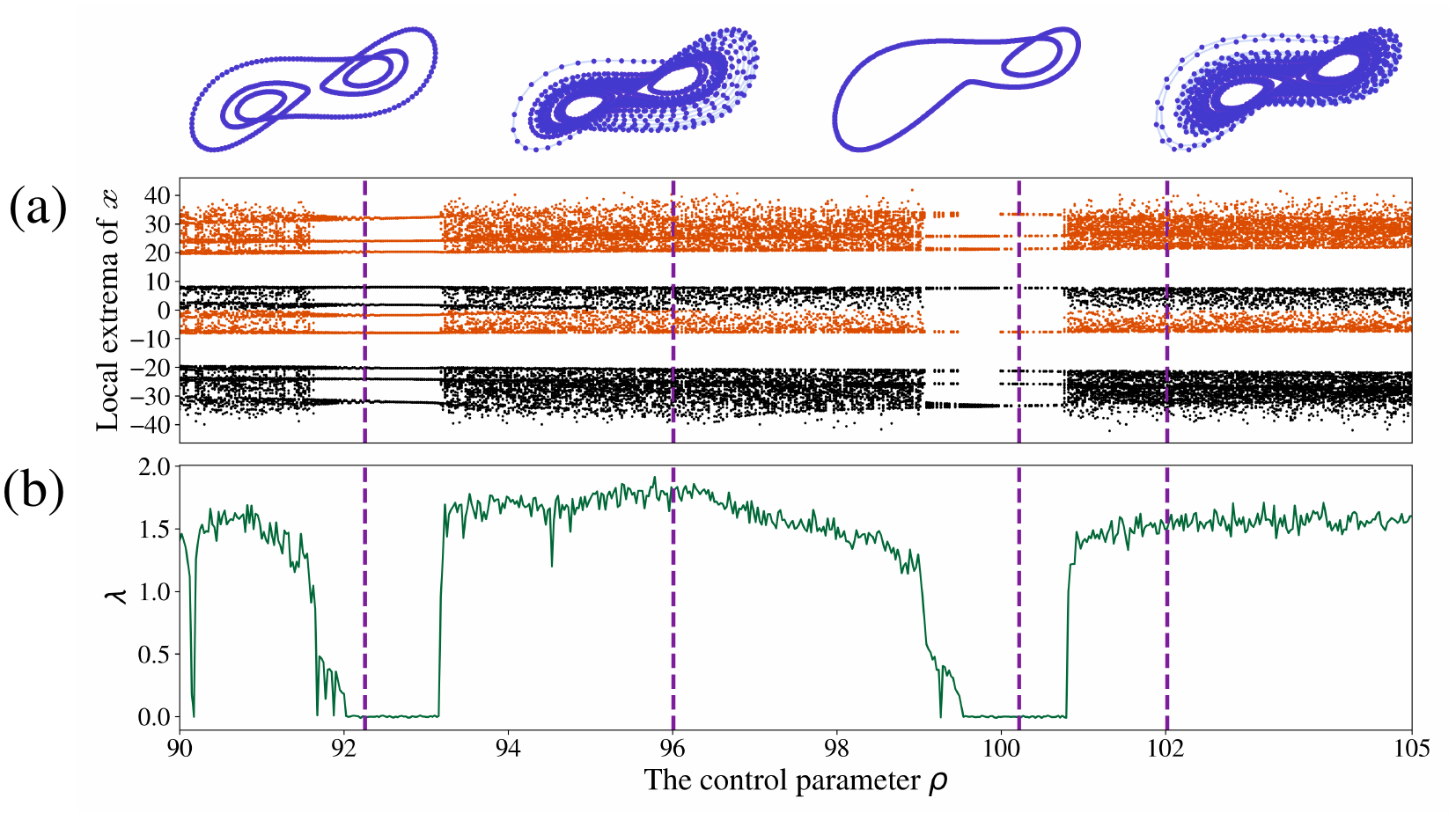}
        \caption{The bifurcation diagram (a) and  the maximum Lyapunov exponent (b) for the Lorenz system with varying 600 $\rho$ parameters between 90 and 105 and $\sigma = 10,\: \beta = 8/3$. 
        On the bifurcation diagram, the \revision{orange} scatter plot represents local maxima while \revision{black} represents local minima of the variable $x$.
        }
        \label{fig:LorenzBifurcLyapunov}
    \end{figure*}

\subsubsection{The maximum Lyapunov exponents}
The second tool we use to investigate the behavior of a system, in this case, a quantitative measure,  is called the maximum Lyapunov exponent. 
{It can be described} as the mean rate of exponential divergence or convergence of two neighboring beginning points in the phase space of a dynamical system. 
Specifically, the quantity is defined as follows. 

\begin{definition}
Given a dynamical system $\dot{Z}=F(Z)$ with $n-$dimensional phase space, we consider two neighboring trajectories $Z(t)$ and $Z(t)+\delta(t)$, where $\delta(t)$ is a vector with infinitesimal initial length. 
The maximum Lyapunov exponent of the system is a number $\lambda$, if it exists, such that $||\delta(t)||\thickapprox||\delta(0)||e^{\lambda t}$.
Specifically,
\begin{equation*}
\lambda = \lim_{t\to\infty} \frac{1}{t}\log \frac{\|\delta(t)\|}{\|\delta(0)\|},
\end{equation*}
where $\|\cdot \|$ is the Euclidean norm. 
Moreover, the system is called chaotic if $\lambda>0$, periodic if $\lambda=0$ and stable if $\lambda <0$. 

\end{definition}

\revision{
We use the maximum Lyapunov exponent as our ground truth measurement for whether a dynamical system is chaotic, periodic, or stable at a given parameter value, and make use of two different algorithms for its computation. 
First, we use the Benettin algorithm~\cite{benettin1980lyapunov} to compute the maximum Lyapunov exponent in the case that we are assumed to have a model of the dynamical system. 
In particular, analytical calculations (when possible) are computationally expensive since they require the Jacobian matrix.

However, as is often the case in practice, when the model of the dynamical system is not available, we turn to numerical approximations. 
In our experiments, the maximum Lyapunov exponent is numerically computed  using Rösenstein's algorithm~\cite{rosenstein1993practical}. 
In this algorithm, the maximum Lyapunov exponent is calculated using the slopes of the average divergence curves that result from linear fits of a scaling region. 
However, human input is required for the results of this procedure to be a reliable approximation for the true Lyapunov exponent. 
That being said, this algorithm is still commonly used in practice and thus it is useful to explore the relationships between the Rösenstein results and the TDA methods introduced here.
The experiments described in Sec.~\ref{sec:Experiments} involve far too many time series to be able to perform detailed human checking, so we only use these results as an approximation and show further information on the challenges and quality of approximation for our experiments in Appx.~\ref{Appx:DifferentLyapAlgo}. 
}

\revision{
Returning to the example of Fig.~\ref{fig:LorenzBifurcLyapunov}(b), we see the maximum Lyapunov exponent, labeled $\lambda$, on the bottom of the figure. 
We calculated this maximum Lyapunov exponent from the model by using Benettin algorithm implementation in \texttt{DynamicalSystems.jl}. 
Note that there are still errors in the Benettin algorithm, since, for example, in the regions around 92.1 and 100.1, $\lambda$ approaches 0, but due to approximation errors does not stay at exactly 0. 
}
Further, we can see that regions around these parameter values sometimes have non-zero Lyapunov exponent despite having visually periodic behavior in the bifurcation diagram. 
We will further investigate this situation in the experiment section, Sec.~\ref{sec:Experiments}.

\subsection{Topological Data Analysis}
\label{ssec:TDA}

A modern tool for measuring the shape of data, particularly point cloud data of the form we have sampled for a given dynamical system $X = [x_1,\cdots,x_N]$, is that of persistent homology and related ideas. 
This construction comes from the field of topological data analysis \cite{Dey2021,Oudot2017,Munch2017}, and has many variants, including Betti curves, persistence diagrams, and persistence barcodes. 
In this paper, we will focus on the CROCKER plot \cite{crocker-Lori2015}, which is a relative of persistence in the case of a parameter stream of data, and provide the background needed for its definition.

\subsubsection{Persistent homology}
 
The idea behind persistent homology is that we can encode information about a fixed shape using homology \cite{Hatcher, Munkres2}, and extend this idea to develop a filtration of shapes where we can measure the changing homology. 
In this section, we give a view of the basics but direct the interested reader to relevant texts \cite{Dey2021,Oudot2017} for further specifics. 

The data we study in this work are point clouds, i.e.~a finite discrete sets of points $X \subset \R^{N}$.
We can develop a filtration of simplicial complexes based on this point cloud using the Vietoris-Rips complex, defined as follows.

\begin{definition}
    Given a point cloud $ X $, the \emph{Vietoris-Rips complex} is defined to be the simplicial complex  whose simplices are built on vertices that are at most $\varepsilon$  apart,
	\[ 
	R_{\varepsilon}(X) = \{\sigma \subset X\mid d(x,y)\leq \varepsilon, \text{ for all }
	x,y\in \sigma\}. 
	\]
	When $X$ is clear from context, we denote the complex simply by $R_\e$. 
\end{definition}
For a fixed $\e$, we get a fixed simplicial complex $R_\e(X)$. 
These complexes have the additional property that for $\e \leq \e'$, $R_\e(X) \subseteq R_{\e'}(X)$. 
The resulting collection of Vietoris-Rips complexes for a sequence of proximity parameters, $ \epsilon_0 \leq \epsilon_1\leq \cdots \leq \e_s $, is the filtration
\begin{equation*}
	R_{\epsilon_0} \subseteq R_{\epsilon_1} \subseteq \cdots \subseteq R_{\e_s}.
\end{equation*}
An example of this construction can be seen in Fig.~\ref{fig:RipsVarying}(a-d).

	\begin{figure}%
    \centering
    \includegraphics[width=\linewidth]{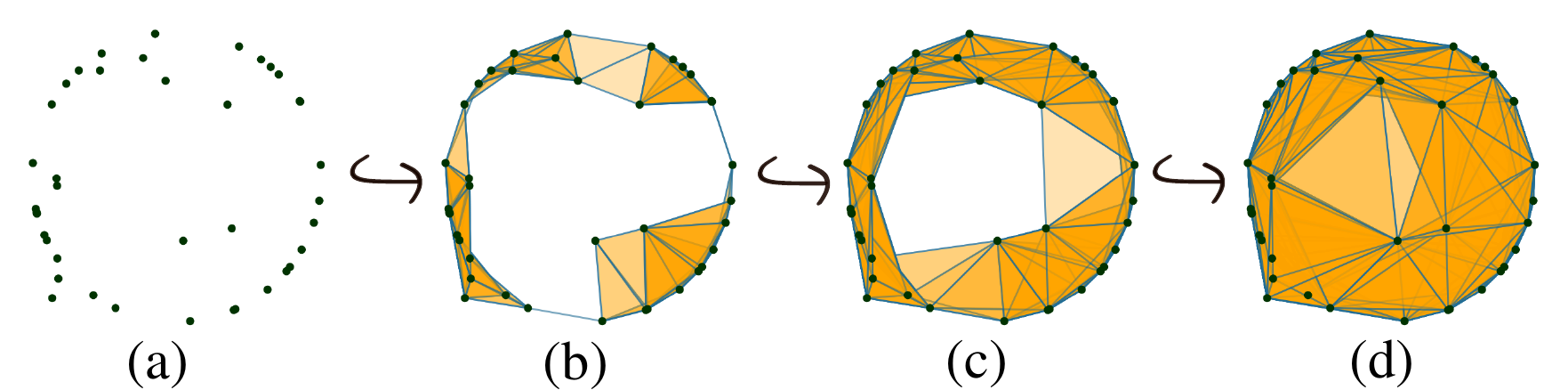}
    \caption{The Vietoris-Rips complex for four choices of an increasing proximity parameter $\varepsilon$. 
    The 1-dimensional Betti numbers, which count the number of loops in each complex, are 0, 1, 1 and 0 in order from left (a) to right (d).
    }\label{fig:RipsVarying}
    \end{figure}   	
To measure the changing shape, we turn to homology.  
For a given simplicial complex $K$, the $k$-dimensional homology $H_k(K)$ is a vector space representing $k$-dimensional structure in the complex. 
The lowest dimensions come with the most intuition; namely that 0-dimensional homology measures connected components, 1-dimensional homology measures loops, and 2-dimensional homology measures voids. 
We direct the interested reader to classical texts\cite{Hatcher, Munkres2} for the full specifics on homology. 
For our purposes, the important information is that the basis of the vector space $H_k(K)$ has one element per $k$-dimensional structure. 
The dimension of $H_k(K)$ is called the $k^{\text{th}}$ Betti number, and will play an important role in the next section.
So, in the case of 1-dimensional homology in the example of Fig.~\ref{fig:RipsVarying}, the $1^{\text{st}}$ Betti numbers for each simplicial complex shown are 0, 1, 1, and 0, respectively as only the second and third complexes, (b) and (c) have a loop.  

Of course, in the example of Fig.~\ref{fig:RipsVarying}, we can see that for certain choices of $\e$ proximity parameters, the structure of the Vietoris-Rips complex $R_\e(X)$ is a good approximation for the structure seen in the point cloud; namely that the points appear to be sampled from a circle. 
However \textit{a priori},  we have no working knowledge to determine the right choice of $\e$ in advance, especially when we happen to not be working with a point cloud in low dimensions like $n=2$. 
To this end, we turn to persistent homology, which encodes information about all values of $\e$ at once, allowing us to measure the durability of topological features over the changing parameter.
    
The additional tool we use is that the vector spaces $H_k(R_\e(X))$ also come with linear maps induced by the inclusions $R_\e(X) \hookrightarrow R_{\e'}(X)$. 
While we again leave specifics to classical texts \cite{Hatcher, Munkres2}, the result is a collection of vector spaces and maps 
\begin{equation*}
 H_{k}(R_{\epsilon_0}) \rightarrow H_{k}(R_{\epsilon_1}) \rightarrow \cdots \rightarrow H_{k}(R_{\e_s})
\end{equation*}
called a persistence module. 

Further results in algebra mean that the persistence module can be decomposed into interval modules, which effectively give information about when features measured by homology appear and disappear over the module. 
Details can be found in \cite{Oudot2017}. 
\revision{This data can be uniquely represented through a collection of pairs $(\e_{b},\e_{d})$ in a decomposition of the module, where each pair represents the parameter values for which a homological feature appeared at time $b$ and disappeared at time $d$.} 
To visualize this information, we turn to the persistence diagram, where an example  can be seen in Fig.~\ref{fig:bettibarcode}. %

\begin{figure*}%
    	\centering
    	\includegraphics[width=\linewidth]{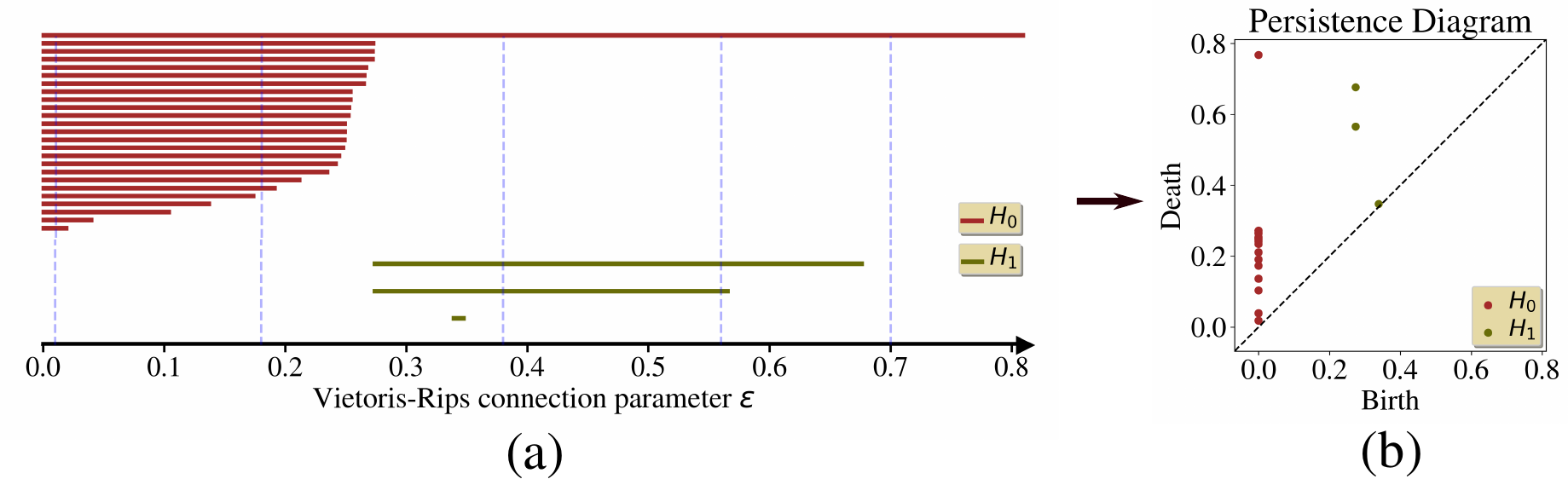}
    	\caption{The persistence barcode with Vietoris-Rips complex for chosen epsilon parameters (a) and corresponding persistence diagram (b). The Vietoris-Rips filtration parameters $[0.01, 0.18, 0.38, 0.56, 0.7]$ and 1-dimensional Betti numbers corresponding to given filtration parameter $\varepsilon$ is $[0, 0, 2, 2, 0]$ .
    	}\label{fig:bettibarcode}
\end{figure*}
	
\begin{definition}
    A \emph{persistence diagram} is a finite collection of off-diagonal points
    $$D = \{(b,d) \in \mathbb{R}^2 \mid b< d\}$$
    where $b$ and $d$ are the  birth and death time of a feature in the persistence module. 
    We call with $d-b$ the \emph{lifetime} of the feature. 
    
    The \emph{empty diagram} is the diagram with no off-diagonal points, $D_{\emptyset} = \{ \}.$
\end{definition}

A persistence barcode encodes the same information as the persistence diagram, but instead, is a collection of horizontal line segments as in Fig.~\ref{fig:bettibarcode}. 
We place the homology generators on the vertical axis (where order does not matter) whereas the horizontal axis represents the life span of each homology class in terms of the parameter $ \varepsilon $. 
When we draw the vertical line at a particular $ \varepsilon_i $, the number of intersecting line segments in barcodes is the dimension of the corresponding homology group, i.e.~the Betti number for that parameter $\varepsilon_i$.  
In Fig. \ref{fig:bettibarcode}, one can see barcodes together with the Vietoris-Rips complex corresponding to several choices of  $ \varepsilon $.

One of the reasons these persistence diagrams are so useful is the availability of metrics for their comparison. 
We focus here on the family of Wasserstein distances for persistence diagram, which effectively measure how easily the points in two diagrams can be matched up. 
To handle diagrams of different sizes, we allow for unmatched points in the two diagrams. 

\begin{definition}\label{defn:Wasserstein}
    A \emph{partial matching} between two diagrams $D_1$ and $D_2$ is a bijection $\mathbb{M}: S_1 \rightarrow S_2$
    where $S_1\subset D_1$ and $S_2\subset D_2$. 
    Two points $\alpha_1\in D_1$ and $\alpha_2\in D_2$ are called matched when $\alpha_2 = \mathbb{M}(\alpha_1)$ while an element $\alpha_1\in (D_1\backslash S_1) \cup (D_2\backslash S_2) =: U_{\mathbb{M}}$ is called unmatched.
    The \emph{Wasserstein distance} between a pair of diagrams $ D_1$ and $ D_2$ is defined as 
    {\small\begin{align*}%
     W_{q}(D_1, D_2)=\inf _{\mathbb{M}}\left(\sum_{(\alpha_1,\alpha_2) \in \mathbb{M}}\|\alpha_1-\alpha_2\|_{\infty}^{q}  + \frac{1}{2^q}\sum_{\alpha \in U_{\mathbb{M}} }|\alpha_d-\alpha_b|^{q}  \right)^{1 / q},  
    \end{align*}}%
    where {the infimum is taken over} all possible partial matches $\mathbb{M} \subset D_1\times D_2$.
\end{definition}
 See the example of Fig.~\ref{fig:Matching}, where two diagrams are overlaid. 
 {In Fig.~\ref{fig:Matching}(b)}, a poor choice of matching results in a high matching cost, while the matching shown at right Fig.~\ref{fig:Matching}(c) achieves the infimum and thus the Wasserstein distance between the diagrams.

	\begin{figure}%
		\centering
		\includegraphics[width=\linewidth]{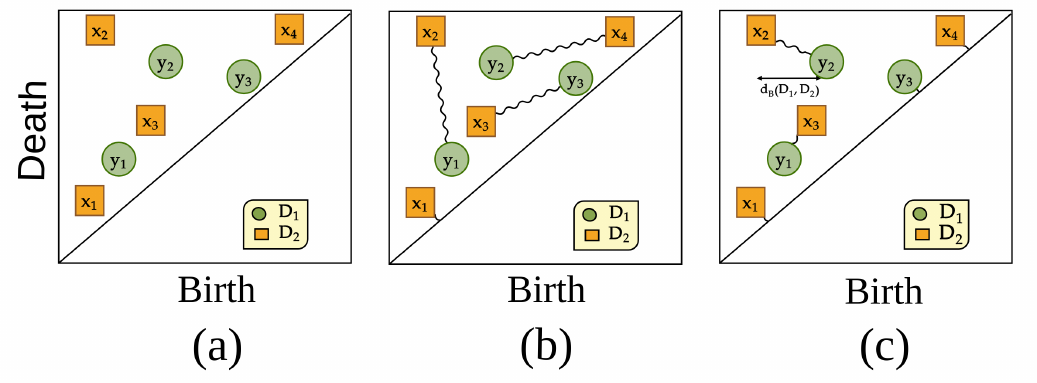}
		\caption{Consider several choices of matchings used in Defn.~ \ref{defn:Wasserstein} for two diagrams $D_1$ and $D_2$ shown overlaid in (a). 
		One option for a matching is shown in (b), but the best matching based on the cost is given in (c). }
		\label{fig:Matching}
	\end{figure}

\subsubsection{Betti curves and Betti vectors}

Despite losing some information in the process, another useful representation of the changing shape is to simply watch the changing rank of the homology rather than tracking the full homology along the filtration. 
Recall that the $k^{\text{th}}$ Betti number is the dimension of the $k^{\text{th}}$ homology group, denoted $\beta_k(K) = \dim(H_k(K))$. 
Then in our setting, the Betti curve is the function encoding the Betti number for each choice of $\e$ proximity parameter in the Vietoris-Rips complex.

\begin{definition}{\label{defn:Betticurve}}
Fix a point cloud $X$ and denote the Vietoris-Rips complex at proximity parameter $\e$ by $R_\e(X)$. 
Then the Betti curve is a function  
\begin{equation*}
\begin{matrix}
B_X: & \R & \longrightarrow  & \N\\
 & \e & \longmapsto & \beta_k(R_{\epsilon}(X)). 
\end{matrix}
\end{equation*}

\end{definition}

We further have a discretized version of the Betti curve, called the Betti vector, which encodes the information from the Betti curve at a fixed sequence of $\e$ values. 

\begin{definition}{\label{defn:Bettivector}}
    Let $P=\{\epsilon_0, \epsilon_1,\dots,\epsilon_{s}\}$ be a partition of the interval $(\epsilon_0=0, \epsilon_{s})$. 
    The $k$-dimensional \emph{Betti vector} is defined as the ordered sequence of the $k$-dimensional Betti numbers, that is
    $$
    Bv_k(X;P) = 
    (\beta_k(R_{\epsilon_0}),\beta_k(R_{\epsilon_1}),\dots, \beta_k(R_{\e_s})).
    $$
\end{definition}

One particularly useful property of the Betti curve viewpoint over the persistence diagram is that we have access to a norm.  
In particular, for any $ p\in(0,\infty] $, the $ L_p $ norm of a real-valued function $ f $ is  
    \[ ||f||_p = \left( \int |f(x)|^p dx \right)^{1/p}. \]

Even though it is known that the Betti curves are unstable~\cite{johnson2021instability, chung2022persistence}, there is a close relationship between the norm of the Betti curve and the Wasserstein distance between a given diagram and the empty diagram. 
Specifically, this is encompassed in the following proposition which is proved in Appendix~\ref{App:BettiNorm}.

\begin{restatable}{proposition}{RelCurveWasser}
\label{prop:RelCurveWasser}
Let $D$ be a persistence diagram with maximum death time $d_{max}$, and $\mathbb{I} = (0,d_{max})\subset\mathbb{R}$. Then, the link between Betti curves and Wasserstein distance is given by  
        \[ 
        \int_{\mathbb{I}} B_{D}(s)\ ds = 2 \cdot W_1(D,D_{\emptyset})  . 
        \]
\end{restatable}

\revision{
This proposition implies that even if they are not stable (which should be thought of in analogue to Lipschitz functions), the Betti curves are at least continuous with respect to variations in the input data. 
This result comes from combining the above proposition with the the stability theorem for Wasserstein distance \cite{Cohen-Steiner2010,Skraba2020}.
}

\subsubsection{CROCKER plots}
\label{subsect:time-evolving}

The data we analyze here is not a single point cloud, but instead one point cloud realization {$\{X_\eta\}_{\eta \in \mathcal{A}}$}  per parameter value of the parameter space of interest {$\mathcal{A} = \{ \eta_0,\cdots, \eta_m\}$} in the dynamical system. 
Thus in the framework above, we get one persistence diagram per parameter value {$\{D_\eta\}_{\eta \in \mathcal{A}}$.}
There have been several methods proposed for accommodating analysis of data of this form. 
These include multiparameter persistence \cite{Carlsson2009}, and vineyards \cite{CohenSteiner2006}.  
However, mathematical limitations and obfuscating interpretations make these difficult to work with. 
For this reason, we look at one of the most accessible representations of this data, the Contour Realization Of Computed $k$-dimensional hole Evolution in the Vietoris-Rips complex\cite{crocker-Lori2015}, also known as a CROCKER\footnote{The naming convention comes from the association to Betti number information, but might  \href{https://en.wikipedia.org/wiki/Betty_Crocker}{only be obvious to those who regularly shop in US grocery stores}.}  plot.

While much of the original work \cite{crocker-Lori2015,crocker-Lori2019} was focused on the special case of time-varying point clouds (where the $i^{\text{th}}$ point in the cloud \revision{$X_\eta$}  is related to the $i^{\text{th}}$ point in cloud {$X_{\eta + \e}$}, e.g.~\cite{multiparameterrank}), there is no need for such a requirement. 
Indeed all we need is a sequence of Betti vectors, as in Defn.~\ref{defn:Bettivector},  over a changing parameter.

\begin{definition}
Fix a  partition of proximity parameters $P = \{ 0, \e_1, \cdots, \e_s\}$ and a  partition of the parameter space of interest 
{$\mathcal{A} = \{ \eta_0,\cdots, \eta_m\}$.}
Denote the collection of point clouds by 
$\cX = \{X_{\eta_i} \}_{i=0}^m$.
Then the $k^{\text{th}}$ CROCKER plot is the discrete function  
\begin{equation*}
    CP(\cX)[\eta_i,\e_j] = \beta_k (R_{\e_j}(X_{\eta_i})).
\end{equation*}

\end{definition} 

Note that viewed as a matrix, the CROCKER plot is the Betti vectors for each $\alpha$ value stacked horizontally, 
\begin{equation*}
 CP(\mathcal{\mathcal{X}}) = 
 \begin{pmatrix}
 | & | & & |\\
    Bv(X_{\eta_0};P) & Bv(X_{\eta_1};P) &  \cdots &  Bv(X_{\eta_m};P) \\
 | & | & & |
 \end{pmatrix}.
\end{equation*}

\begin{figure*}%
    \centering
    \includegraphics[width = \linewidth]{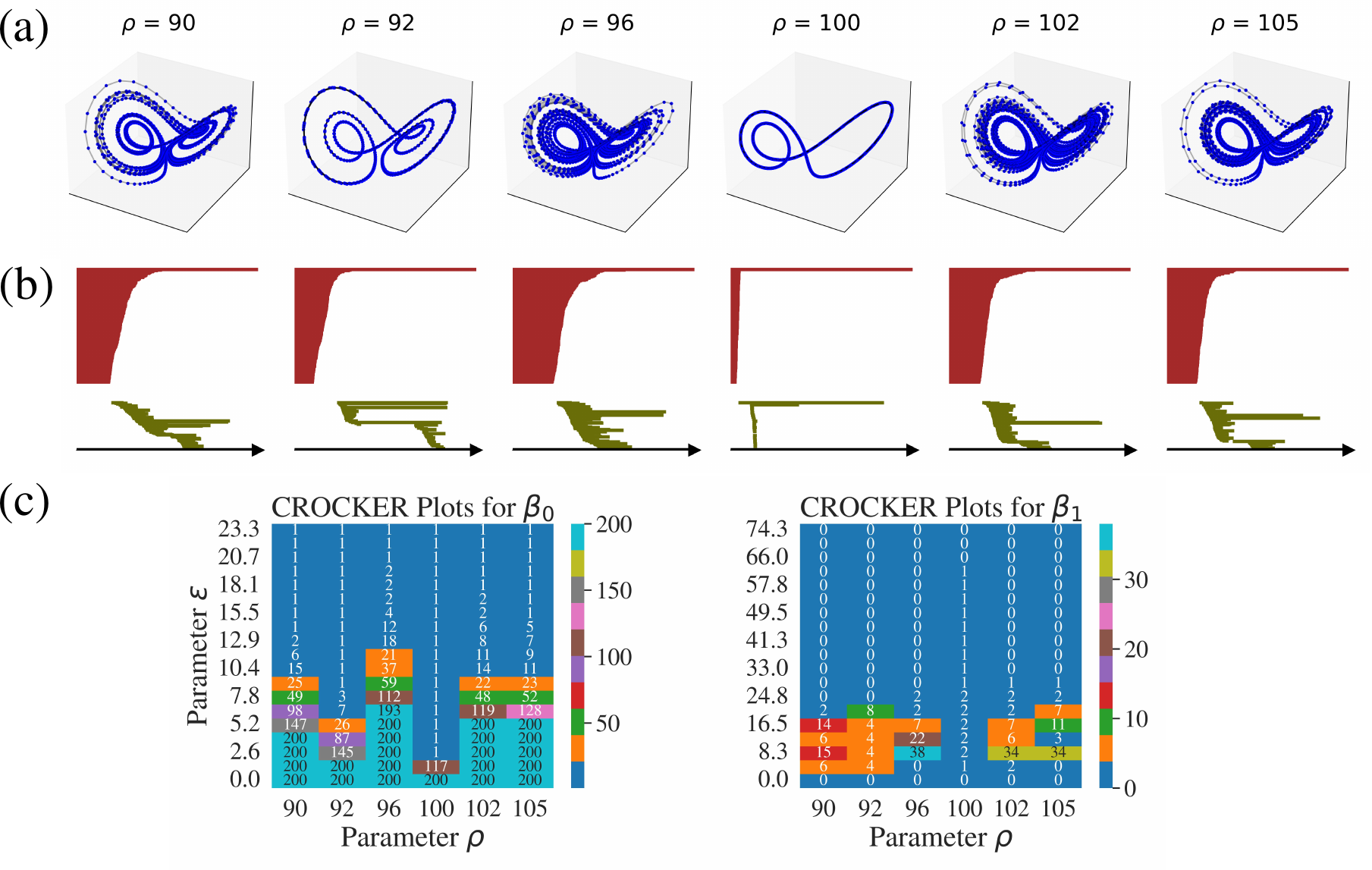}
    \caption{
    An example showing the CROCKER plots (bottom row (c); dimensions 0 and 1) for a coarse partitioning of the $\rho$ parameter for the Lorenz system in the first row (a).
    The corresponding persistence barcodes are shown in the middle row (b).
    }
    \label{fig:LorenzEx}
\end{figure*}

Consider the example of Fig.~\ref{fig:LorenzEx}, where we have six realizations of the Lorenz system for choices of the different $\rho$ in {Fig.~\ref{fig:LorenzEx}(a)}. 
The barcode for each is shown in the next row Fig.~\ref{fig:LorenzEx}(b). 
The $0$ and $1$-dimensional CROCKER plots are shown at the bottom Fig.~\ref{fig:LorenzEx}(c) as the heat maps. 
Each of the 6 columns is the Betti vector for the corresponding realization. 
The high values of Betti numbers for the low $\e$ values come from the many short bars present in the noisy systems. 
Overarching circular structure, which shows up as long bars in the barcode, appear as positive values that stretch up into the higher $\e$ values.

The goal of this paper is to use these CROCKER plots to measure  the changing topological behaviour of dynamics of nonlinear systems while varying both a system parameter and the Rips proximity parameter. 
Unlike tools used in classic bifurcation analysis, the only needed input is the reconstructed point cloud of the attractor. 
In the following section, we will show examples for how this plot can be used to find bifurcations in the behavior of dynamical systems. 

\section{Experiments and Results}
\label{sec:Experiments}

After giving specifics of the choices made in the method, in this section we show the pipeline as applied to two nonlinear dynamical systems: the Rössler system and Lorenz system.
We then give similar results for a longer list of systems, with further details of each in Appendix~\ref{App:DynamicalSystem}. 

\subsection{Specifics for the method}

Our experiments will test this pipeline for the CROCKER plots as an analysis tool for bifurcation in dynamical systems. 
While the method itself can be seen in the example Fig.~\ref{fig:LorenzEx}, we provide additional specifics and choices in this section. 

\revision{
 The dynamical systems were simulated using \texttt{DynamicalSystems.jl} ~\cite{juliaDynamicalSystems} package in julia~\cite{bezanson2017julia}. We also calculate the maximum Lyapunov exponents based on Benettin algorithm in same library. We then fetch the states of dynamical system as a point cloud for the python programming language to continue on CROCKER tasks.
}
Due to the high density of points in the resulting point clouds and computational expense of persistence diagrams, we subsampled each realization $X_\eta$ using a greedy sub-sampling algorithm \cite{greedyalgo} resulting in 500 points per point cloud.
We then compute the $0$- and $1$-dimensional persistence diagrams $D_\eta$ using \texttt{ripser} \cite{ripser}.

We compute all simulations prior to choosing the partition $P$ so as to ensure that we cover a range that includes all points in the persistence diagrams. 
Specifically, if $M$ is the maximum death time of any point in all the $k$-dimensional persistence diagrams (excluding the $\infty$-point present in all $0$-dimensional diagrams), we evenly split the interval $[0,M]$ into 100 pieces. 
Thus $P = \{0, \tfrac{1}{100} \cdot M, \tfrac{2}{100} \cdot M, \cdots, M\}$. 
With this fixed choice of partition, we can then use the persistence diagrams to get the Betti vectors using the equation 
$\beta_k(R_\e(X_\eta)) = \# \{ (b,d) \in D_\eta \mid b \leq \e < d\}$.
Note that in theory, one could speed up computation by computing the Betti vectors for a given partition directly without computing the full persistence diagrams first; however all software we are aware of returns the Betti vectors by computing the full diagram first.

\subsection{Rössler System}	

 The Rössler system is given by the equations
    \begin{equation*}
		\dot{x} =  -y -z,\quad
		\dot{y} =  x + ay,\quad
		\dot{z} =  b+ z(x-c).
	\end{equation*}
For this study, we fixed parameters $b = 2$ and $c = 4 $. 
We then vary the control parameter $a$ for 600 equally spaced values between 0.37 and 0.43. We used a sampling rate of 15 Hz for 1000 seconds and the initial conditions \revision{$[x, y, z] = [-0.4,0.6,1]$}. After solving the system, we retain the last 170 seconds to avoid transients.

\begin{figure*}%
    \centering
    \includegraphics[width=.8\linewidth]{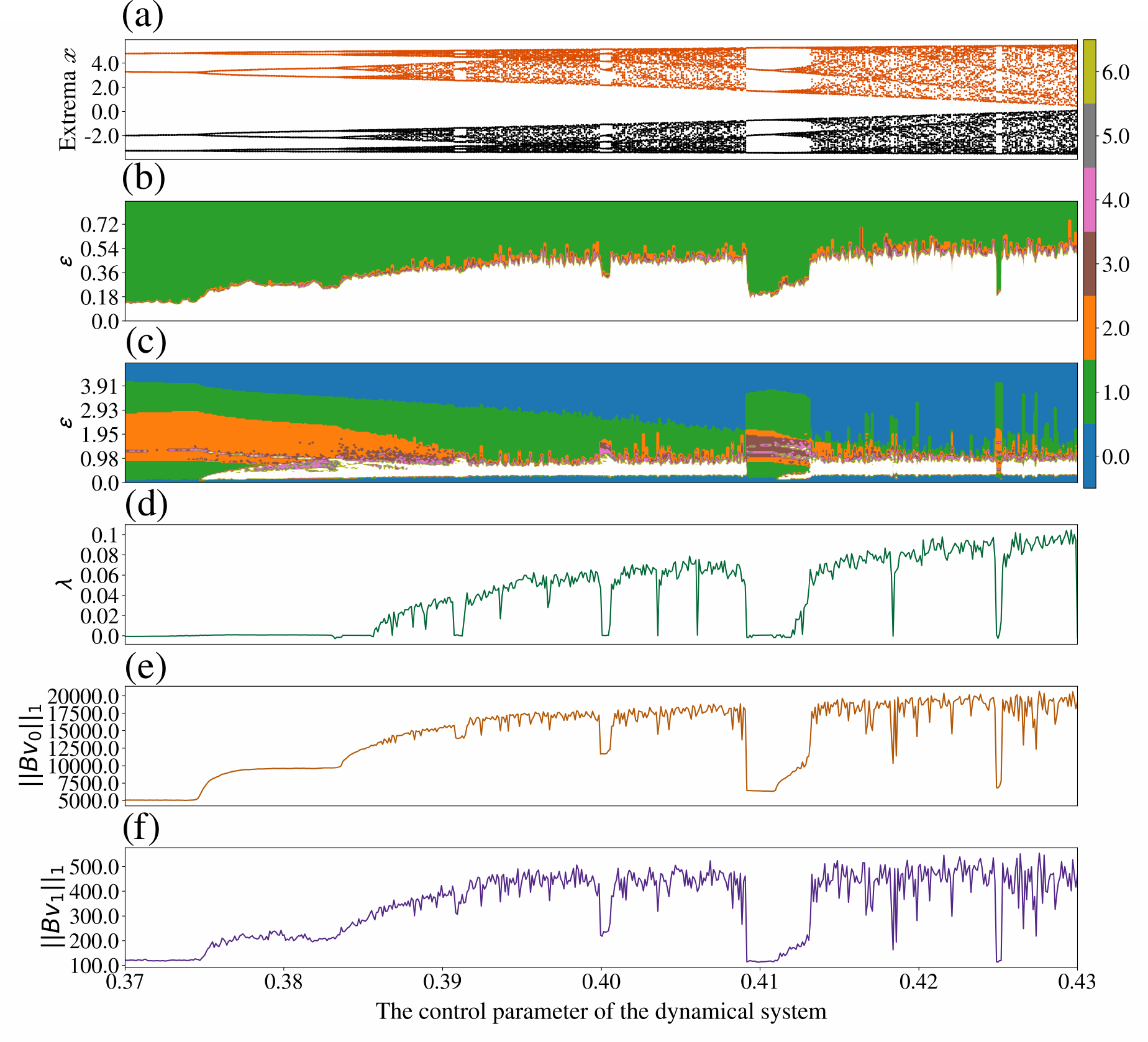}
    \caption{The bifurcation diagram (a) and Lyapunov exponent (d) are shown for the Rössler system varying the control parameter $a$. 
    The 0- and 1-CROCKER plots (b and c respectively) and $L_1$ norm of each Betti vector (e and f) are also shown.}
    \label{fig:BifurcationRossler}
\end{figure*}

In Fig.~\ref{fig:BifurcationRossler}(a), we show the bifurcation diagram for the Rössler system varying parameter $a$ with 600 equally-spaced values between 0.37 and 0.43 with the other parameters fixed. 
The local maxima and minima values of the system variable $x$ are colored \revision{orange} and \revision{black}, respectively. 
Below this are the CROCKER plots for homology dimensions 0 (Fig.~\ref{fig:BifurcationRossler}(b)) and 1 (Fig.~\ref{fig:BifurcationRossler}(c)). 
The value of the Betti number is shown by color; white regions are values for which \revision{$\beta_k >6$}. 
In these cases, those regions are largely viewed as noise. 
For instance, the large white region in the 0-CROCKER is caused by small values of $\e$ resulting in many individual connected components in $R_\e(X)$. 
Indeed, for small enough $\e$, this is one component per point in the cloud. 
The timing for these components merging up is a result of density of the sampling of the realization of the dynamical system. 
As such, the white region in the 0-CROCKER plot here is largely a measure of our subsampling required to compute persistence on a smaller point cloud.
Periodic regions include multiple runs along the attractor, thus increasing the density of coverage and resulting in earlier mergings in the 0-dimensional diagram.

Thanks to these visualization tools, we can qualitatively compare the CROCKER plots with the bifurcation diagram.
Note that most markedly around the bifurcation that occurs around $a = 0.41$, there is a clear shift in the CROCKER plots as well as in the bifurcation diagram. 
This occurs both in the 0- and 1-dimensional diagrams. 
In the 0-dimensional CROCKER plot, there is a drop in the maximum $\e$ (vertical axis) at which the Betti curve becomes a constant 1\footnote{Every 0-dimensional diagram for a point cloud has an infinite bar representing the first connected component which is born and is considered to live for every value of $\e$. Thus the 0-dimensional CROCKER plots always have a region of 1 at their maximum $\e$ boundary.}.
The 1-dimensional CROCKER plot has a marked shift in the maximum $\e$ value for which the Betti curve is positive, meaning this persistence diagram has a long lived bar. 
Further, the thick brown region ($\beta_1 = 3$) implies there are likely two additional long lived bars in the persistence diagrams in this area. 
Moreover, when we vertically look at 1-dimensional CROCKER, the system has two prominent circular shapes in the system parameters interval between 0.37 and 0.39. For instance, one of them has a lifetime of approximately 4 ($\epsilon_{b}\approx 0.2,\: \epsilon_{d} \approx 4.02$) while the other one has a lifetime of 2 ($\epsilon_{b}\approx 0.9,\: \epsilon_{d} \approx 2.9$) at the system parameter 0.37. 
On the other hand, when we horizontally drive the system parameter to the right sides between the parameters 0.39 and 0.42, these 2 loops will change into a circular shape since one of them disappeared. 

There is an even further  simplified version of viewing the data, namely the $L_1$ norm of the Betti vector.
For partition $P = \{\e_0,\cdots,\e_s\}$, this is given as
\begin{equation}
    \| Bv_k(X;P)\|_1 = \sum_{\forall i} \beta_k(R_{\e_i}).
\end{equation}
These graphs are shown in parts (e) and (f) of Fig.~\ref{fig:BifurcationRossler}. 
Again, we can see qualitative similarities, in particular when comparing the graphs of the $L_1$ norms with the {maximum} Lyapunov exponent graph $\lambda$ \revision{obtained from Benettin algorithm}. 

For a more quantitative test, we compute the Pearson correlation coefficient between the \revision{maximum} Lyapunov exponent and the $L_1$ norms for 0- and 1-dimensional information. 
Full details of this computation are given in Appendix \ref{App:correlations}, but the idea is that values close to 1 imply a strong positive linear correlation.
For our cases, the computed Pearson coefficient values  were \revision{0.88 and 0.85}, respectively. 
This means that there is a strongly positive correlation between them.

\subsection{Lorenz System}
We run the same test on  Lorenz system which consists of three ordinary differential equations, 
    \begin{equation*}
        \dot{x} = \sigma(y-x),\quad
        \dot{y} = x(\rho-z)-y,\quad
        \dot{z} = xy-\beta z.
    \end{equation*}
Here the parameters $\sigma = 10$ and $\beta = 8/3$ are fixed, while the control parameter $\rho$ is varied across 600 equally spaced values in the interval $[90,105]$. 
The solutions are simulated using a sampling rate of 100 Hz for 100 seconds with the initial conditions {$[x, y, z] = [10^{-10},0,1]$}. 
After solving the system, we take the last 20 seconds to avoid transients.
    
\begin{figure*}%
    \centering
    \includegraphics[width=.8\linewidth]{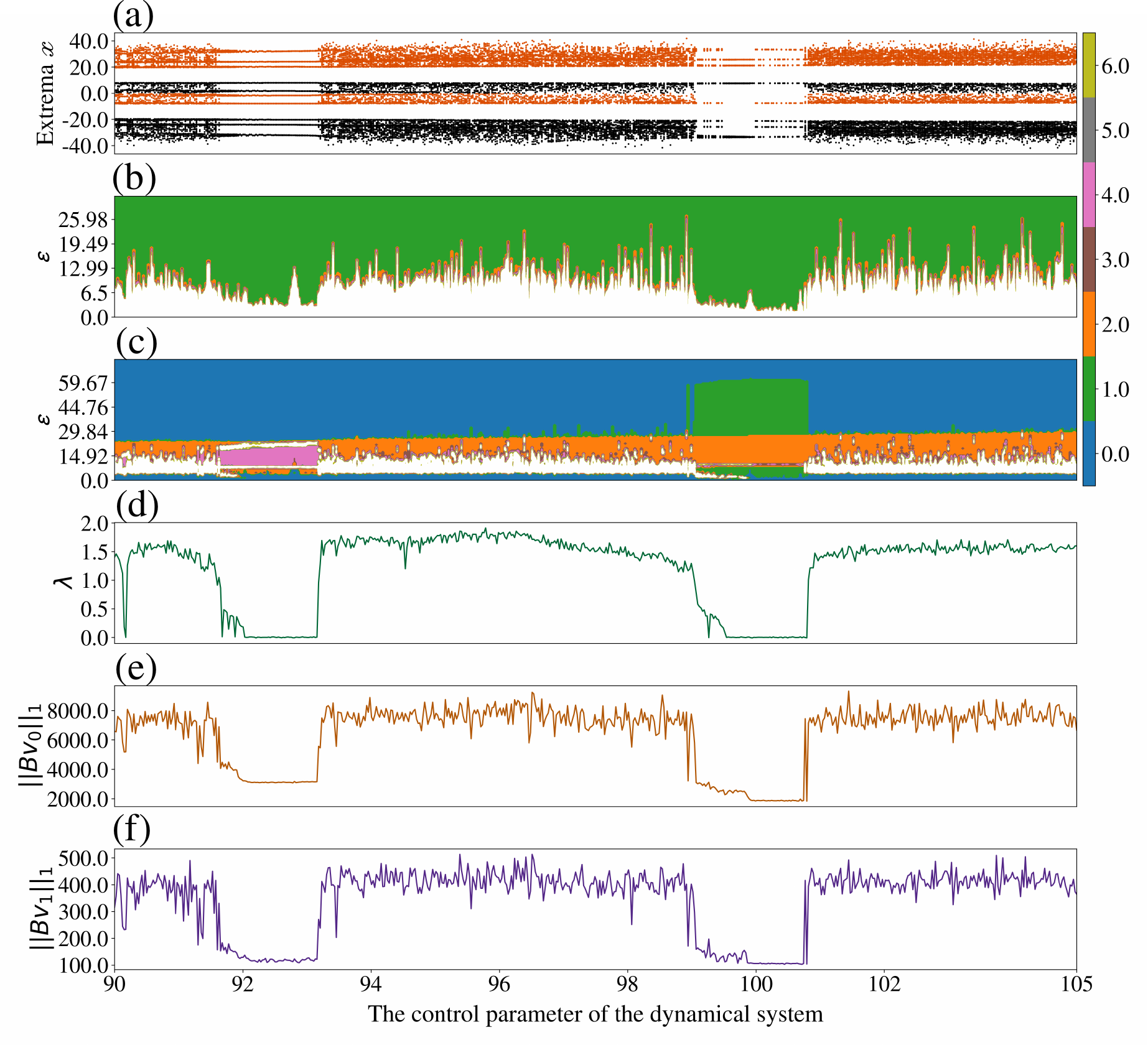}
    \caption{The bifurcation diagram (a), the 0- and 1- dimensional CROCKER plots (b and c, respectively), the maximum Lyapunov exponents (d) and for each CROCKER plots , the $L_1$ norm of each Betti vector (e and f) corresponding to varying the control parameter $\rho$ for the Lorenz system.}
    \label{fig:BifurcationLorenz}
\end{figure*}   
    
In Fig.~\ref{fig:BifurcationLorenz}, as in the R\"{o}ssler system, Lorenz's CROCKER and $L_1$ norms exhibit similar characteristics to the bifurcation diagram and the {maximum} Lyapunov exponent.
We also note that there is a clear relationship between the {maximum} Lyapunov exponent and the $L_1$ norm of the CROCKER vectors as seen on the Fig.~\ref{fig:BifurcationLorenz}(d-f). 
The computed Pearson coefficient values between the {maximum} Lyapunov exponent, and the $L_1$ norms for 0- and 1-dimensional {Betti vectors} were 0.93 in both cases.

However, when we examine two different system parameters in the same system, some differences emerge in the case of 1-dimensional CROCKER. 
In particular as shown in Fig.~\ref{fig:BifurcationLorenzParticular}, consider the parameters $\rho$ around 92.5 and 100. 
While in the 0-dimensional CROCKER, there is not an obvious difference  between the two regions. 
However, in this particular case, the 1-dimensional CROCKER  shows a stark contrast. 
For example, there are 4 noticeably persistent points around $\rho=92.5$, which can be seen in the thick pink region corresponding to $\beta_1=4$ in that area of the 1-dimensional CROCKER. 
Similarly, there is an extremely long lived persistence bar around $\rho=100$ seen in the thick green region which extends farther up into the $\e$ parameter of the vertical axis.
These regions were chosen because they encompass regions of parameter space where the Lyapunov exponent, as seen in Fig.~\ref{fig:BifurcationLorenz}, sometimes has non-zero approximations computed. 
This means that the CROCKER plot can potentially provide a more fine-grained separation than the {maximum} Lyapunov exponent alone. 
     
\begin{figure*}%
    \centering
    \includegraphics[width=\linewidth]{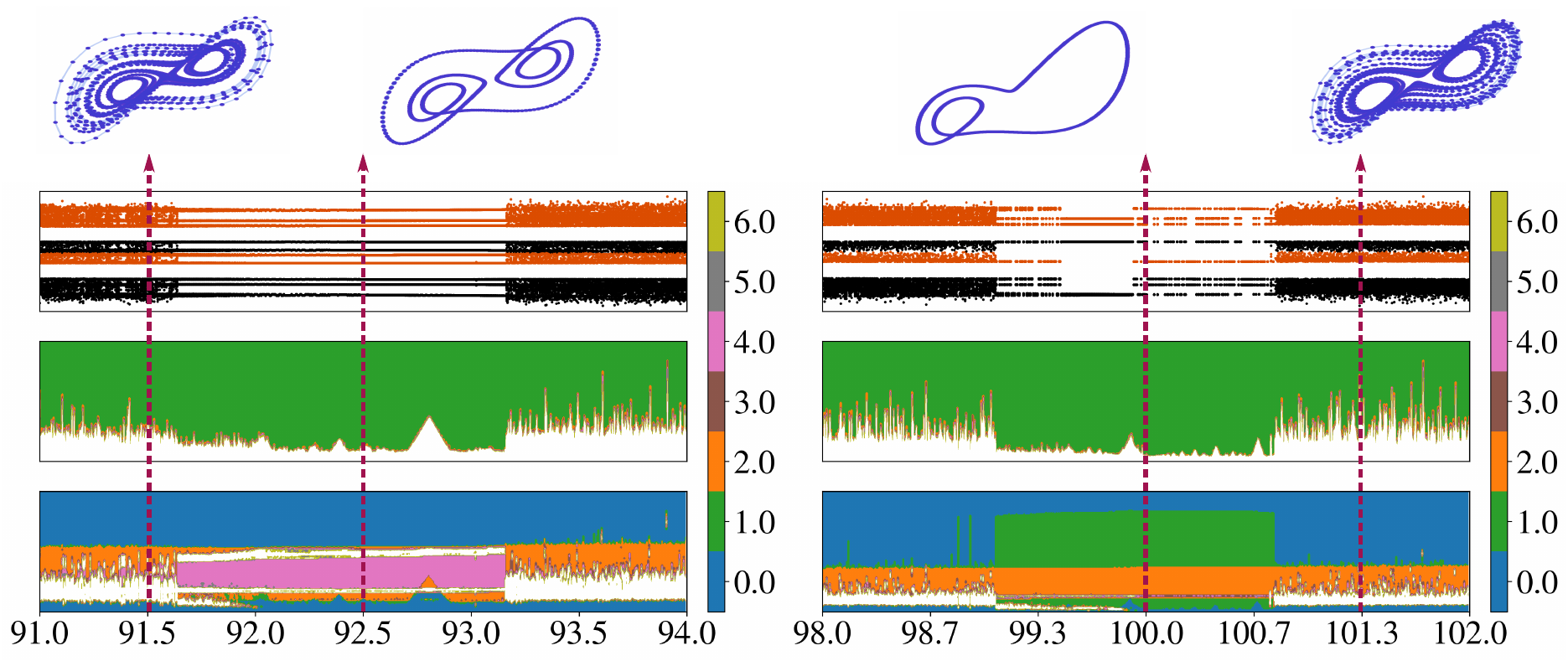}
    \caption{\revision{Zoomed in versions of Fig.~\ref{fig:BifurcationLorenz} showing that the 1-dimensional CROCKER plot can detect changes in behavior related to different sorts of looping structures which were harder to detect using the Lyapunov exponent shown in Fig.~\ref{fig:LorenzBifurcLyapunov}.}}
    \label{fig:BifurcationLorenzParticular}
\end{figure*}

\subsection{Correlation between the maximum Lyapunov exponents and the $L_1$ norm of CROCKER plots}

We have shown a close relationship in terms of the Pearson correlation coefficient for both the R\"{o}ssler and Lorenz systems; to see if this behavior persists, in this section we show similar tests for a larger collection of dynamical systems. 
Specifics for the systems and their simulation information are given in  Appendix~\ref{App:DynamicalSystem}. 
In Table~\ref{table:correlation}, we show both the Pearson $r$  and Spearman $\rho$ correlation coefficients with subscripts denoting the dimension of CROCKER plots used \revision{in each algorithm}. 
Details of $r$ and $\rho$  can be found in Appendix~\ref{App:correlations}. 

From Table~\ref{table:correlation}, we can conclude, at least experimentally, that the correlation between the maximal Lyapunov exponent and 1- dimensional $L_1$ norm of CROCKER is a strong positive linear.
\revision{
The correlation coefficients between the maximum Lyapunov exponent came from the Benettin algorithm and $L_1$ norm of Betti vectors is greater than the one obtained from R\"{o}senstein algorithm.
It shows that the $L_1$ norms of Betti vectors includes more information than R\"{o}senstein algorithm when we do not have the full equations of the model.
}
For completeness, the appendix (Fig.~\ref{fig:allcrockers}) includes the  bifurcation diagrams and CROCKER plots for the dynamical systems given in Table~\ref{table:correlation} in Appendix~\ref{App:DynamicalSystem}.

\begin{table*}%
\centering
\newlength{\width}
\width12mm
\caption{
\revision{
The Spearman $\rho$ and Pearson $r$ correlation between the maximum Lyapunov exponent, obtained from R\"{o}senstein and Benettin algorithm, and $L_1$ norm of Betti vectors for each dimension 0 and 1.
}
}\label{table:correlation}
\begin{tabular}{p{4\width} p{0.9\width} p{0.9\width} p{0.9\width} p{0.9\width} p{0.9\width}  p{0.9\width} p{0.9\width} p{0.9\width} p{0.9\width}}
&\multicolumn{4}{c}{\textbf{R\"{o}senstein Algorithm}}& & \multicolumn{4}{c}{\textbf{Benettin Algorithm}}  \\
\cline{2-5} \cline{7-10}
\textbf{Systems} & $\rho_0$ & $r_0$ &  $\rho_1$ & $r_1$ & & $\rho_0$ & $r_0$ &  $\rho_1$ & $r_1$ \\
\hline 
Lorenz & $0.679$ & $0.854$ & $0.647$ & $\mathbf{0.857}$ & & $0.657$ & $\mathbf{0.936}$ & $0.622$ & $0.934$ \\
Rössler & $\mathbf{0.863}$ & $0.856$ & $0.797$ & $0.823$ & & $\mathbf{0.884}$ & $0.88$ & $0.807$ & $0.847$ \\
Coupled Lorenz Rössler & $0.924$ & $\mathbf{0.925}$ & $0.887$ & $0.878$ & & $0.912$ & $\mathbf{0.937}$ & $0.874$ & $0.892$ \\
Complex butterfly & $0.934$ & $0.942$ & $0.932$ & $\mathbf{0.945}$ & & $0.927$ & $0.948$ & $0.912$ & $\mathbf{0.953}$\\
Hadley circulation & $0.805$ & $0.861$ & $0.832$ & $\mathbf{0.871}$ & & $0.816$ & $0.884$ & $0.841$ & $\mathbf{0.892}$\\
Moore-Spiegel  & $0.592$ & $\mathbf{0.841}$ & $0.593$ & $0.808$ & & $0.572$ & $\mathbf{0.863}$ & $0.563$ & $0.845$\\
Halvorsens  & $0.737$ & $0.857$ & $0.745$ & $\mathbf{0.888}$ & & $0.731$ & $0.865$ & $0.71$ & $\mathbf{0.897}$\\
Burke-Shaw  & $0.805$ & $0.843$ & $0.823$ & $\mathbf{0.859}$ & & $0.815$ & $0.833$ & $0.833$ & $\mathbf{0.873}$\\
Rucklidge  & $0.842$ & $\mathbf{0.934}$ & $0.835$ & $0.928$ & & $0.852$ & $\mathbf{0.942}$ & $0.832$ & $0.933$\\
WINDMI & $0.787$ & $0.896$ & $0.789$ & $\mathbf{0.916}$ & & $0.789$ & $0.898$ & $0.802$ & $\mathbf{0.921}$
\end{tabular}

\end{table*}
    
\subsection{Computation Cost}

We further note the computational cost of each method \revision{in the setting where we assume the input data is available but the model is not. 
For this purpose, we first simulate the solution of the dynamical system using the python library \texttt{teaspoon} \cite{teaspoon}. 
We then record the computation time of two methods starting from the input of the available data: (i) $L_1$ norms of Betti vectors and (ii) the maximum Lyapunov exponent computed used R\"{o}senstein algorithm.}
All  computations were performed on a Ubuntu 20.10 desktop with 16 GB RAM, Intel(R) Core(TM) i7-9700 CPU 3.00GHz, and 8 cores using the python language. 
A table of computation times for the dynamical systems defined in Appendix~\ref{App:DynamicalSystem} is given in Table~\ref{Table:ExecutionTime}.
These calculations are done for 600 bifurcation parameter values;  the average and standard deviation of the run times is shown.
Computation for determining $\|Bv_i\|_1$ is time from input of point cloud to output of $L_1$ norm; and so includes subsampling, construction of the simplicial complex filtration, persistence diagram, Betti vector, and computation of the norm. 
Lyapunov code was performed using the slope of the average divergence \revision{curves obtained} from linear fits as described \revision{in the original R\"{o}senstein paper\cite{rosenstein1993practical}. To get comparable execution time, we fixed the number of divergences are calculated to be in the range of the linear region, as first 20 distances in the calculation of the maximum Lyapunov exponent.}
We note that in most cases the run time for the topological representations is comparable with the Lyapunov exponent computation, and in many cases is actually considerably faster. 
These computations were performed largely with out-of-the-box open source software; we suspect that computations can be sped up with more attention paid to code developed specifically for this purpose.

\begin{table}
\centering
\caption{For each system defined in Appendix~\ref{App:DynamicalSystem}, we record the execution time to obtain \revision{the maximum Lyapunov exponents used the R\"{o}senstein algorithm} and the $L_1$ norm of Betti vectors for each particular system parameter.}
\begin{tabular}{p{3.2cm}  p{1.6cm}  p{1.6cm} p{1.6cm}}
\textbf{Systems}          & \hspace{1pt}\textbf{Lyapunov} & \hspace{8.5pt}$\mathbf{\|Bv_{0}\|_1}$ & \hspace{8.5pt}$\mathbf{\|Bv_{1}\|_1}$  \\ \hline
Lorenz                    & \textbf{0.42 $\pm$ 0.47}         & 0.53   $\pm$ 0.56               & 0.51 $\pm$ 0.69          \\
Rossler                   & 1.07 $\pm$ 0.71         & 1.21   $\pm$ 1.86               & \textbf{0.98 $\pm$ 1.08}          \\
Coupled Lorenz Rossler    & \textbf{0.38 $\pm$ 0.35}         & 1.24   $\pm$ 1.88               & 1.12 $\pm$ 2.04          \\
Complex butterfly         & 2.52 $\pm$ 1.15         & 0.36   $\pm$ 0.31               & \textbf{0.33 $\pm$ 0.31}          \\
Hadley circulation        & 2.83 $\pm$ 1.03         & 1.24   $\pm$ 1.51               & \textbf{1.21 $\pm$ 1.49}          \\
Moore-Spiegel             & 2.47 $\pm$ 0.91         & 0.34   $\pm$ 0.34               & \textbf{0.28 $\pm$ 0.12}          \\
Halvorsens                & 2.49 $\pm$ 0.92         & \textbf{1.09   $\pm$ 0.81}               & 1.19 $\pm$ 1.69          \\
Burke-Shaw                & 2.73 $\pm$ 2.31         & \textbf{1.35   $\pm$ 0.92}               & 1.41 $\pm$ 1.39          \\
Rucklidge                 & 2.36 $\pm$ 1.21         & \textbf{0.79   $\pm$ 0.79}               & 0.91 $\pm$ 1.09          \\
WINDMI                    & 1.97 $\pm$ 0.52         & \textbf{0.63   $\pm$ 0.43}               & 0.78 $\pm$ 0.55         
\end{tabular}
\label{Table:ExecutionTime}
\end{table}

\section{Conclusions and future directions}
\label{sec:Conclusions}

In this work, we have begun an investigation of the use of CROCKER plots for bifurcation analysis in dynamical systems. 
We show that in a simple test case, there is clearly a relationship between a representation of change in the system (the Lyapunov exponent) with the structure of the CROCKER plot, as well as with the $L_1$ norm of each Betti vector.
We also proved a relationship between the 1-Wasserstein distance to the empty diagram, and the $L_1$ norm of the Betti vector.

This work, of course, leads to many interesting open questions. 
For starters, all work in this paper has been done in a data-driven manner. 
We suspect more can be done to provide a theoretical connection between the Lyapunov exponent and the $L_1$ norm of the Betti curve. 

Further, while our computation times are \revision{on par} with those for computing Lyapunov, we believe more can be done to improve \revision{the computation} time of the method as shown.  
First, to our knowledge all TDA code available computes Betti vectors by first computing the full persistence diagram, and then reverse engineering the Betti curve. 
Might there be a more direct computation method which provides speedups relative to the desired refinement of the partition?
    
Another potentially useful aspect of the CROCKER method for analyzing chaos as opposed to Lyapunov exponents is that of sensitivity to noise. 
In particular, Lyapunov is notorious for its slow computation and its sensitivity to error and noise. 
On the other hand, persistence comes with a theoretically grounded framework showing stability to noisy systems; i.e.~that the construction of persistence from input data is a 1-Lipschitz procedure. 
The issue to be overcome is that the instability of Betti numbers~\cite{crocker-stack, johnson2021instability} extends to the CROCKER plots. 
That being said, our Prop.~\ref{prop:RelCurveWasser} shows that the construction of Betti curves is a continuous procedure. 
Further work, either data driven or theoretical, is needed to understand whether this instability shows up in the cases of the embedded time series data being studied in this paper.

\section*{Acknowledgments}
The research of IG was supported by a grant program (B\.{I}DEB 2214-A:1059B142000135) from the T\"{U}B\.{I}TAK, Scientific and Technological Research Council of T\"{u}rkiye.
This material is based upon work supported by the Air Force Office of Scientific Research under award number FA9550-22-1-0007.

The authors wish to thank the input of two anonymous reviewers whose feedback helped strengthen the final manuscript. 
The authors also thank Liz Bradley for valuable discussions and pointing our attention to issues with the calculation of maximum Lyapunov exponents.

\section*{Author Declarations}
\subsection*{Conflict of Interest}
The authors have no conflicts to disclose.

\section*{Data Availability}
The data that support the findings of this study are available from the corresponding author upon reasonable request.

\appendix

\revision{    
\section{Benettin vs.~R\"osenstein's Algorithms for Maximal Lyapunov Exponents}
\label{Appx:DifferentLyapAlgo}
In this section, we give further information on the differences between the Benettin and R\"osenstein algorithms for computing maximal Lyapunov exponents. 
For more information on the maximum Lyapunov exponent, we refer the reader to Kantz and Schreiber's book~\cite{Kantz2004} and Parlitz's notes~\cite{Parlitz2016}.

The main difference between the two algorithms is as follows.
The R\"osenstein algorithm can be computed using only a reconstructed attractor from a dynamical system without necessarily having access to the underlying differential equations. 
In this case, the goal is to determine the slope of the average divergence curve and shown in Fig.~\ref{fig:DiffAlgorithLE}(b).

The issue that arises is that manual input is required.  
The approximation for the maximum Lyapunov exponent is the slope of a fit to the initial portion of the curve. 
If we are lucky and choose $k=25$, we see that the fitted yellow line gives a reasonable approximation to the slope of the initial portion. 
However, if we instead choose $k=40$, the fitted line is no longer a reasonable estimate.
The solution is, in fact, to manually check all input time series, however, for the large number of realizations studied in this manuscript, this is unreasonable to do beyond basic spot checking. 
All calculations used in the main manuscript use $k=20$, which was chosen as it seems to give reasonable results. 

On the other hand, if we are in a situation where we have access to the model information, a better choice of Lyapunov computation is Benettin's algorithm. 
This works by evolving the tangent vectors and following the Gram-Schmidt procedure which requires the Jacobian matrix of the system.
For this reason, we treat the Benettin calculation as the ground truth in this manuscript, but provide results with comparison to the R\"osenstein algorithm since this is a commonly used tool in practice when the model is not available.

In Fig.~\ref{fig:DiffAlgorithLE}(a), we have plotted some example computations, where each point represents the maximal Lyapunov exponent for the same realization of the R\"ossler system using the two methods. 
In one case, every computation uses $k=25$ and in the other, $k=40$.
While there is some noise in the approximation, which is to be expected, we see that R\"osenstein is close to a linear fit with the Benettin results. 
The MSE for these lines are $1.3 \times 10^{-6}$ and  $3 \times 10^{-6}$, respectively. 
Note that the outlier point drawn as a red star is the example used to calculate Fig.~\ref{fig:DiffAlgorithLE}(b). 

Note that while the results differ by a multiplicative constant, the only thing that matters for analysis is whether the exponent is zero or non-zero, thus showing that the automated parameter choices (in particular, $k$) we have made are reasonable. 
It will be interesting to see in future work how much dependence there is on the results of the correlation between the TDA results and those using different choices of $k$.

    \begin{figure*}
        \centering
        \includegraphics[width = .7\linewidth]{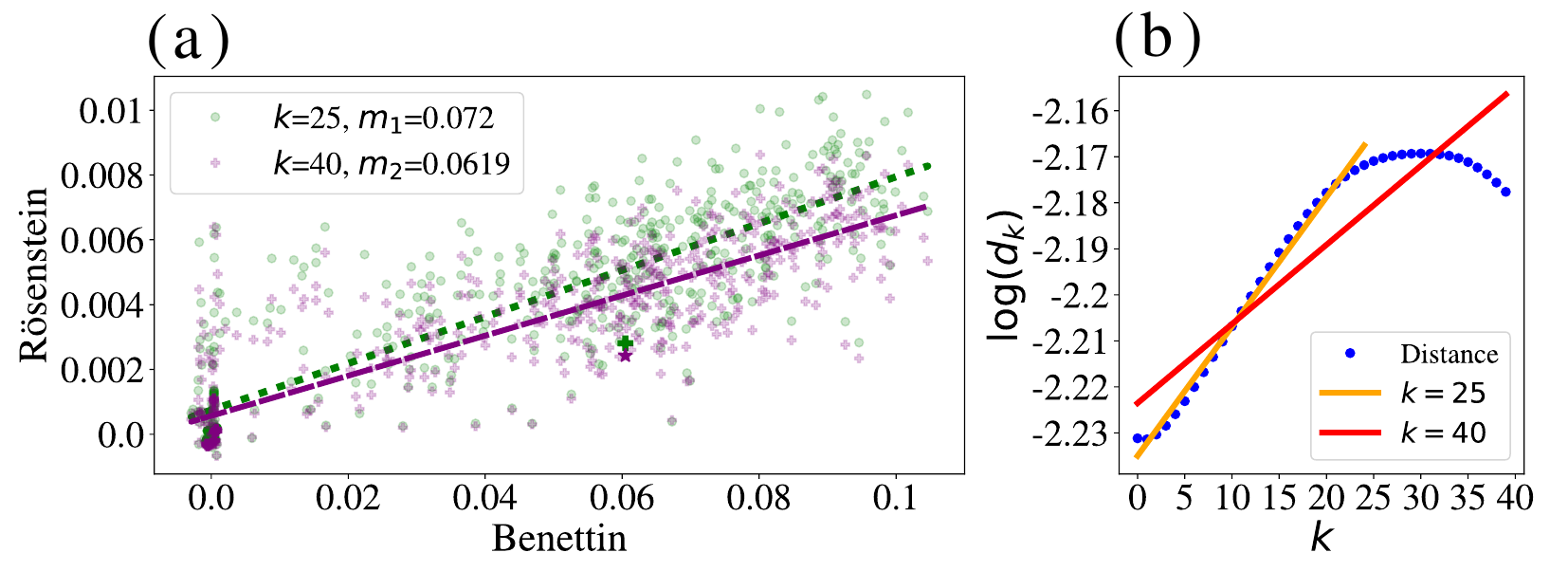}
        \caption{The relationship between two methods, R\"{o}senstein and Benettin algorithm, to calculate the maximum Lyapunov exponent for the different 600 R\"{o}ssler systems. 
        Figure (a) shows nearly the linear relation between them for two different choices of $k$ in the R\"osenstein algorithm. 
        Figure (b) shows the last step of R\"{o}senstein algorithm  to calculate the maximum Lyapunov exponent corresponding to the same system realization shown as the purple star and green plus sign in the scatter plot at left.}
        \label{fig:DiffAlgorithLE}
    \end{figure*}
    
}

\section{Relation between Wasserstein distance and Betti curves }
\label{App:BettiNorm}

In this section of the appendix, we give the proof of Prop.~\ref{prop:RelCurveWasser}, which states that the relationship between the norm of the Betti curve and the 1-Wasserstein distance are related as follows.

\RelCurveWasser*

First, we note that because the Betti curve is an integer valued function, we can rewrite it in terms of indicator functions of each bar in the barcode. 
Specifically, let $ D=\{\alpha_i = (b_i,d_i)\}_{i=0}^{w} $ be a persistence diagram with $|D|=w+1$ points. 
The Betti curve of the persistence diagram 
$ B_D : \mathbb{R} \rightarrow \mathbb{N} $ is given by 
\begin{equation}\label{AppndxEqn:Betti}
 B_D(s) = \sum_{\alpha \in D} \mathds{1}_{\alpha}(s), 
\end{equation}
where the indicator function for a point $\alpha=(b,d)\in D$ is 
\begin{equation*}
\mathds{1}_{\alpha}(s) = 
\begin{cases}
1, \quad b \leq s < d, \\ 
0, \quad \mbox{otherwise}. 
\end{cases}  
\end{equation*}
See, for example, the Betti curve given in Fig.~\ref{fig:barcodescurve_basic}.

On the other hand, we can obtain the topological features by using 1-Wasserstein distance between a given diagram $ D $ and the empty diagram $ D_{\emptyset} $. 
Using Defn.~\ref{defn:Wasserstein}, we have that
    \begin{equation}\label{Wasserstein_1}
       W_1(D,D_{\emptyset}) = \frac{1}{2}\sum_{\alpha \in D} \ell_{\alpha},
    \end{equation}
    where $\ell_{\alpha}=d-b$ is the lifetime of the topological feature $\alpha=(b,d)\in D$.

To prove Prop.~\ref{prop:RelCurveWasser}, we begin with the following lemma, which will be set up to prove the proposition by induction. 
\begin{lemma}
\label{lem:removeBump}

     Let $D=\{\alpha_i = (b_i,d_i)\}_{i=0}^{w} $ be persistence diagram and let $A$ be the maximum value of the Betti curve,  $a = \max_{s \in \R} B_D(s)$.
     Let $\overline{\alpha}_{*}$ be a maximal (with respect to inclusion) connected interval where $B_D$ takes value $A$, i.e.~$B_D(s) = A$ for all $s \in \overline{\alpha}_{*}.$
     Then there is a persistence diagram $\widetilde D$ which contains $\overline \alpha_*$ as a bar, has the same number of bars as $D$, and satisfies 
    $$
    W_1(D,D_\emptyset) = W_1(\tilde{D},D_\emptyset) \quad \mbox{ and } \quad B_D(s) = B_{\tilde{D}}(s). 
    $$
    
\end{lemma}    
    
\begin{proof}
If $\overline \alpha_*$ is already a bar of $D$, we set $\widetilde D = D$ to complete the proof, so we may assume that there is no such bar in $D$. 
This means that there must be two bars corresponding to the endpoints of the interval. 
Then let  $\alpha_1=(b_1,d_1)$ and $\alpha_2=(b_2,d_2)$
be bars in $D$ with 
$b_1<b_2=b_{*}<d_{*}=d_1<d_2$. 
See Fig.~\ref{fig:barcodescurve_basic} for an example of the two cases.
    
Define $\alpha_1'=(b_2,d_1)$ and $\alpha_2'=(b_1,d_2)$ to be a pair of bars with traded endpoints and let 
$\tilde{D} = D\backslash \{\alpha_1, \alpha_2\} \cup \{\alpha_1', \alpha_2'\}$ be the diagram with those bars swapped in. 
Then because the lifetimes satisfy
\begin{align*}
 \ell_{\alpha_1} + \ell_{\alpha_2} &= (d_1-b_1) + (d_2-b_2)\\
 &= (d_1-b_2) + (d_2-b_1)\\
 &= \ell_{\alpha_1'} + \ell_{\alpha_2'}
\end{align*}
we have that
\begin{align*}
    W_1(D,D_\emptyset) = \frac{1}{2}\sum_{\alpha\in D} \ell_{\alpha} 
                    = &\frac{1}{2}\left( \ell_{\alpha_1} + \ell_{\alpha_2} \right) + \frac{1}{2}\sum_{\alpha\in D\backslash\{\alpha_1, \alpha_2\}} \ell_{\alpha},\\
                    = &\frac{1}{2}\left( \ell_{\alpha_1'} + \ell_{\alpha_2'}\right) + \frac{1}{2}\sum_{\alpha\in D\backslash\{\alpha_1, \alpha_2\}} \ell_{\alpha},\\
                    = &  W_1(\tilde{D},D_\emptyset).
\end{align*}

It is easy to check that 
$
\mathds{1}_{\alpha_1}(s) + \mathds{1}_{\alpha_2}(s) 
=
\mathds{1}_{\overline{\alpha}_1}(s) + \mathds{1}_{\overline{\alpha}_2}(s).$
Combining this with Eqn.~\ref{AppndxEqn:Betti}, we have 
    \begin{align*}
    B_D(s) = \sum_{\alpha \in D} \mathds{1}_{\alpha}(s) = & \mathds{1}_{\alpha_1}(s) +\mathds{1}_{\alpha_2}(s) + \sum_{\alpha \in D\backslash\{\alpha_1, \alpha_2\}} \mathds{1}_{\alpha}(s),\\
    = & \mathds{1}_{\overline{\alpha}_1}(s) +\mathds{1}_{\overline{\alpha}_2}(s) + \sum_{\alpha \in D\backslash \{\alpha_1, \alpha_2\}} \mathds{1}_{\alpha}(s),\\
    = & B_{\tilde{D}}(s)
    \end{align*}
    completing the proof.
    \end{proof}

For simple visualization example,  consider Figure~\ref{fig:barcodescurve_basic}. 
In both example barcodes, we have $\alpha_* = [b_2,d_1)$. 
In Fig.~\ref{fig:barcodescurve_basic}(a) the two bars correspond to $\alpha_1$ and $\alpha_2$, and when endpoints are swapped, result in the bars $\alpha_1'$ and $\alpha_2'$ in Fig.~\ref{fig:barcodescurve_basic}(b). 
However, the Betti curve shown in Fig.~\ref{fig:barcodescurve_basic}(c) remains the same for both. 
	\begin{figure}%
		\centering
		\includegraphics[width=\linewidth]{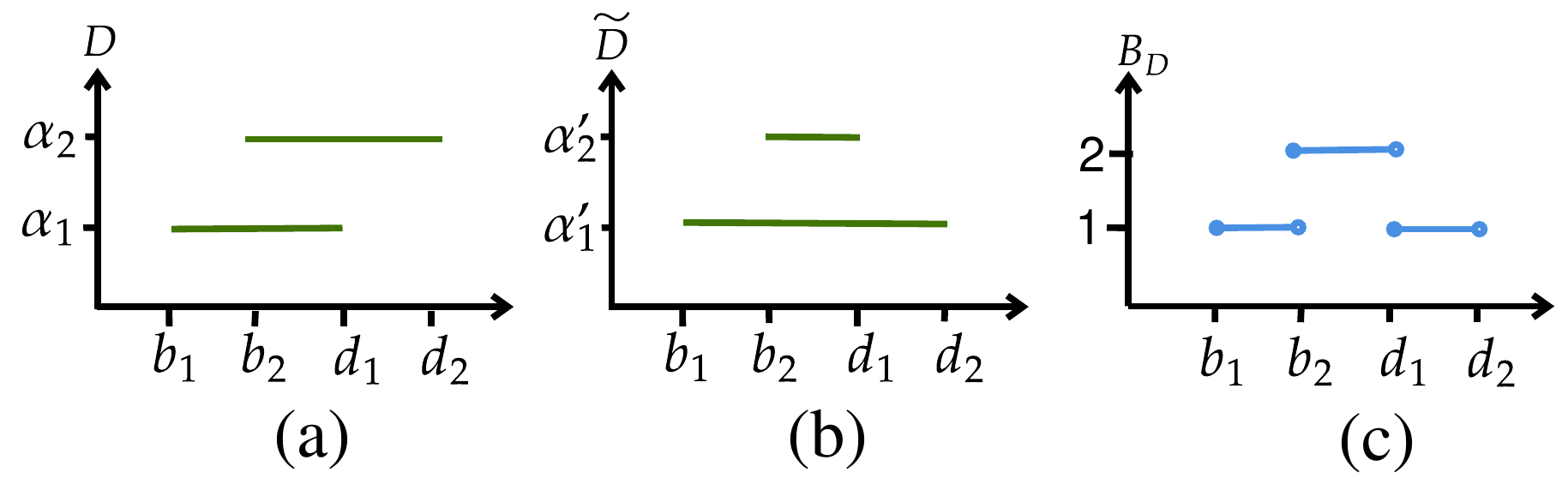}
		\caption{
		An example of two barcodes (panels (a) and (b)) with the same Betti curve (panel (c)) as constructed in Lem.~\ref{lem:removeBump}.
		}
		\label{fig:barcodescurve_basic}
	\end{figure}

We can now use the above lemma in the inductive step of the proof of the main proposition.

\begin{proof}[Proof of Proposition~\ref{prop:RelCurveWasser} ]
We will prove the statement by induction on the number of bars $k$ of a given diagram. 

 Since the case of the empty diagram is vacuously true, we will start with the base case $w=1$ for clarity. 
 This is visualized in Fig.~\ref{basecase}.
 A diagram with a single bar $(s_0,s_1)$ shown in Fig.~\ref{basecase}(a), and the Betti curve is given in Fig.~\ref{basecase}(b). 
 Then the area under the curve is $\int_{\mathbb{I}} B_{D_{1}}(s)\ ds = s_1-s_0$. 
 On the other hand, the 1-Wasserstein distance to the empty diagram, as shown in Fig.~\ref{basecase}(c), the distance is the half lifetime of the topological feature, thus  $\tfrac {1}{2}(s_0,s_1)$. 
 So,
\[
\int_{\mathbb{I}} B_{D_{1}}(s)\ ds 
= (s_1-s_0)\cdot1 = 2\cdot \frac{s_1-s_0}{2} = 2 \cdot W_1(D_{1},D_{\emptyset}) \ .
\]
    
	\begin{figure}%
		\centering
		\includegraphics[width=\linewidth]{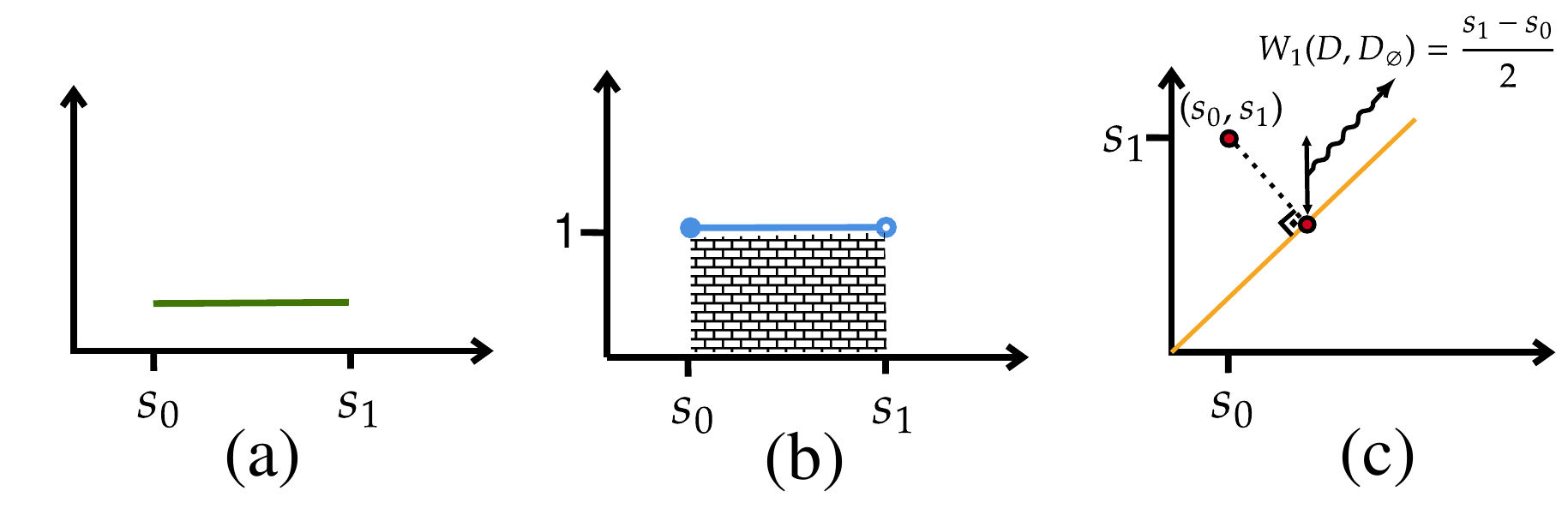}
		\caption{(a) Persistence barcodes, (b) Betti curve and (c) persistence diagram with Wasserstein distances for the example of a diagram with a single point. }\label{basecase}
	\end{figure}

Now we assume the statement is true for a diagram with $w$ bars.     
Assume that we have a persistence diagram $D$ with $w+1$ bars and let $\widetilde D$ be the diagram obtained from Lem.~\ref{lem:removeBump} with bar $\alpha_{*} = (b_*,d_*) \in \widetilde D$. 
Let $\widetilde D_w$ be the diagram $\widetilde D$ with that bar removed, i.e.~$\widetilde D_{w} \cup \{ \alpha_*\} = \widetilde D$ and note that $|D_w| = w$.
Then we have
\begin{align*}
    \int_{\mathbb{I}} B_{D}(s)\ ds  
    = & \int_{\mathbb{I}} B_{\widetilde{D}}(s)\ ds 
        &\mbox{ (Lem.~\ref{lem:removeBump})}\\    
    = &\int_{\mathbb{I}} B_{\widetilde{D}_{w}}(s)\ ds 
        + \int_{\mathbb{I}} \mathds{1}_{\alpha_*}(s)\ ds&\\
     = & 2\cdot W_1(\widetilde{D}_w, D_\emptyset)  + (\widetilde{d}_{*}-\widetilde{b}_{*})& \text{(induction)}\\
     = & 2\cdot W_1(\widetilde{D}, D_\emptyset) & \text{ (Defn.~\ref{Wasserstein_1})}\\
     = & 2\cdot W_1(D, \emptyset)&\mbox{ (from Lemma~\ref{lem:removeBump})}
\end{align*}
completing the proof.
    \end{proof}

\section{Pearson and Spearman Correlations}\label{App:correlations}
The Pearson correlation coefficient $r \in[-1,1]$ measures the linear correlation of two series $x$ and $y$. 
The Pearson correlation coefficient for two datasets is calculated as
\begin{equation}
\label{eq:Pearson}
    r=\frac{\sum_{i=1}^{n}\left(x_{i}-\bar{x}\right)\left(y_{i}-\bar{y}\right)}{\sqrt{\sum_{i=1}^{n}\left(x_{i}-\bar{x}\right)^{2}} \sqrt{\sum_{i=1}^{n}\left(y_{i}-\bar{y}\right)^{2}}} .
\end{equation}
    
    The coefficient $r$ indicates  a perfect positive or negative linear correlation when $r =1$ or $r=-1$, respectively, while $r=0$ represents no linear correlation. 
    However Pearson correlation is limited in that it can only discover linear associations. 
    
    To investigate nonlinear associations, one may consider Spearman's correlation which is the non-parametric version of the Pearson correlation.    
    Spearman's correlation coefficient $\rho\in[-1,1]$  is also calculated using Eq.~\eqref{eq:Pearson} with the ordinal ranking of the variables $x$ and $y$ instead of their numerical values. This substitution allows for detecting nonlinear correlation trends to be represented as long as the correlation is monotonic.

\section{Dynamical System Details}
\label{App:DynamicalSystem}

In this section we give details on the various dynamical systems simulated for the experiments from Sec.~\ref{sec:Experiments}. 
All simulations were performed using the python package \texttt{teaspoon} \cite{teaspoon} with default parameters other than those noted.

    \subsection{Lorenz}
    The Lorenz system consists of three ordinary differential equations referred to as Lorenz equations:
    \begin{equation*}
        \dot{x} = \sigma(y-x),\quad
        \dot{y} = x(\rho-z)-y,\quad
        \dot{z} = xy-\beta z,
    \end{equation*}
    where the  parameters $\sigma = 10,\: \beta = 8/3$  are fixed; 
    the control parameter $\rho$ is varied across 600 equally spaced values between 90 and 105;
    and the system is simulated with a sampling rate of 100 Hz for 100 seconds using the initial condition ${[x, y, z]} = [10^{-10},0,1]$. 
    After solving the system, we take only the last 20 seconds to avoid transients.
    
    \subsection{Rössler}
    
     The Rössler system is defined as follows
    	\begin{equation*}
		\dot{x} =  -y -z,\quad
		\dot{y} =  x + ay,\quad
		\dot{z} =  b+ z(x-c), 
	\end{equation*}
    where the parameters $b = 2,\: c = 4 $ are fixed; 
    the control parameter $a$ is varied across 600 equally spaced  values between 0.37 and 0.43;
    and the system is simulated with a sampling rate of 15 Hz for 1000 seconds using the initial condition ${[x, y, z]} = [-0.4,0.6,1]$. 
    After solving the system, we take only the last 170 seconds to avoid transients.
    
    \subsection{Coupled Lorenz Rössler}
    
    The Coupled Lorenz-Rössler system defined as follows
    \begin{align*}
    \dot{x_1} &=-y_{1}-z_{1}+k_{1}\left(x_{2}-x_{1}\right),  & \dot{x_2} &=\sigma\left(y_{2}-x_{2}\right), \\
    \dot{y_1} &=x_{1}+a_{2} y_{1}+k_{2}\left(y_{2}-y_{1}\right), & \dot{y_2} &=\gamma x_{2}-y_{2}-x_{2} z_{2}, \\
    \dot{z_1} &=b_{2}+z_{1}\left(x_{1}-c_{2}\right)+k_{3}\left(z_{2}-z_{1}\right), & \dot{z_2} &=x_{2} y_{2}-b_{1} z_{2},\\
    \end{align*}
    where the parameters $b_{1}=8/3,\: b_{2}=0.2,\: c_{2}=5.7,\: k_{1}=0.1,\: k_{2}=0.1$,\: $k_{3}=0.1,\: {\gamma}=28, \: \sigma=10$ are fixed; 
    the control parameter $a$ is varied across 600 equally spaced  values between 0.3 and 0.5;
    and the system is simulated with a sampling rate of 50 Hz for 500 seconds using the initial condition {$[x_1, y_1, z_1, x_2, y_2, z_2] = [0.1, 0.1, 0.1, 0, 0, 0]$}. 
    After solving the system, we take only the last 30 seconds to avoid transients.

    \subsection{Complex butterfly}
    
     The complex butterfly system is defined as follows
    	\begin{equation*}
    	\dot{x}=a(y-x), \quad
        \dot{y}=2 \operatorname{sgn}(x), \quad
        \dot{z}=|x|-1,
	\end{equation*}
    where the control parameter $a$ is varied across 600 equally spaced  values between 0.10 and 0.60;
    and the system is simulated with a sampling rate of 10 Hz for 1000 seconds using  the initial condition ${[x, y, z]} = [0.2, 0, 0]$. 
    After solving the system, we take only the last 500 seconds to avoid transients.

    \subsection{Hadley circulation}
    
     The Hadley circulation system is defined as follows
    	\begin{align*}
    		\dot{x} = & -y^2 -z^2 - ax + aF\\
    		\dot{y} = & xy -bxz-y+G\\
    		\dot{z} = & bxy+xz-z, 
    	\end{align*}
    where the parameters $b = 4,\: F = 8,\: G = 1 $ are fixed;
    the control parameter $a$ is varied across 600 equally spaced  values between 0.20 and 0.25;
    and the system is simulated with a sampling rate of 50 Hz for 500 seconds using the initial condition ${[x, y, z]} = [-10, 0, 37]$. 
    After solving the system, we take only the last 80 seconds to avoid transients.

    \subsection{Moore-Spiegel Oscilator}
    
     The Moore-Spiegel oscilator system is defined as follows
    	\begin{equation*}
    		\dot{x} =  y,\quad
    		\dot{y} =  z,\quad
    		\dot{z} =  -z-(T-{R}+Rx^{2})y-Tx, \quad
	    \end{equation*}
    where the parameter $R = 20$ is fixed;
    the control parameter $T$ is varied across 600 equally spaced  values between 7.0 and 8.0;
    and the system is simulated with a sampling rate of 100 Hz for 500 seconds using the initial condition ${[x, y, z]} = [0.2, 0.2, 0.2]$. 
    After solving the system, we take only the last 10 seconds to avoid transients.    
    
    \subsection{Halvorsens cyclically symmetric attractor}
    
     The Halvorsens cyclically symmetric attractor is defined as follows
    	\begin{align*}
    		\dot{x} = &  -ax-by-cz-y^{2}\\
    		\dot{y} = & -ay-bz-cx-z^2 \\
    		\dot{z} = & -az-bx-cy-x^{2}, 
    	\end{align*}
    where the parameters $b = 4,\: c = 4 $ are fixed;
    the control parameter $a$ is varied across 600 equally spaced  values between 1.40 and 1.85;
    and the system is simulated  with a sampling rate of 200 Hz for 200 seconds using the initial condition ${[x, y, z]} = [-5, 0, 0]$. After solving the system, we take only the last 25 seconds to avoid transients.       
    
    \subsection{Burke-Shaw attractor}
    
     The Burke-Shaw system is defined as follows
    	\begin{equation*}
    		\dot{x} =  -s(x+y),\quad
    		\dot{y} =  -y-sxz,\quad
    		\dot{z} =  sxz + V, \quad
    	\end{equation*}
    where the parameter $V = 10 $ is fixed;
    the control parameter $s$ is varied across 600 equally spaced  values between 9.0 and 13.0;
    and the system is simulated  with a sampling rate of 200 Hz for 500 seconds using the initial condition ${[x, y, z]} = [0.6,0,0]$. 
    After solving the system, we take only the last 25 seconds to avoid transients.

    \subsection{Rucklidge attractor}
    
     The Rucklidge system is defined as follows
    	\begin{equation*}
    		\dot{x} =  -kx+\lambda y-yz,\quad
    		\dot{y} =  x,\quad
    		\dot{z} =  -z + y^2, 
    	\end{equation*}
    where the parameter $\lambda = 6.7 $ is fixed; the control parameter $k$ is varied across 600 equally spaced  values between 1.0 and 1.7;
    and the system is simulated  with a sampling rate of 50 Hz for 1000 seconds using the initial condition ${[x, y, z]} = [1,0,4.5]$. 
    After solving the system, we take only the last 100 seconds to avoid transients.

    \subsection{WINDMI}
    
     The WINDMI system is defined as follows
    	\begin{equation*}
    		\dot{x} =  y,\quad
    		\dot{y} =  z,\quad
    		\dot{z} =  -az-y+b-e^x,
    	\end{equation*}
    where the parameter $b = 2.5 $ is fixed; 
    the control parameter $a$ is varied across 600 equally spaced  values between 0.7 and 1.0;
    and the system is simulated  with a sampling rate of 20 Hz for 1000 seconds using the initial condition ${[x, y, z]}= [1,0,4.5]$. 
    After solving the system, we take only the last 250 seconds to avoid transients.

\section{CROCKER plots for other dynamical systems}
\label{Appx:AllCROCKERplots}
In this section, we give an additional figure (Fig.~\ref{fig:allcrockers}) detailing similar results to those  seen in Sec.~\ref{sec:Experiments} for the remainder of the systems mentioned in Appendix~\ref{App:DynamicalSystem}.
    \begin{figure*}
        \centering
        \includegraphics[width = \textwidth]{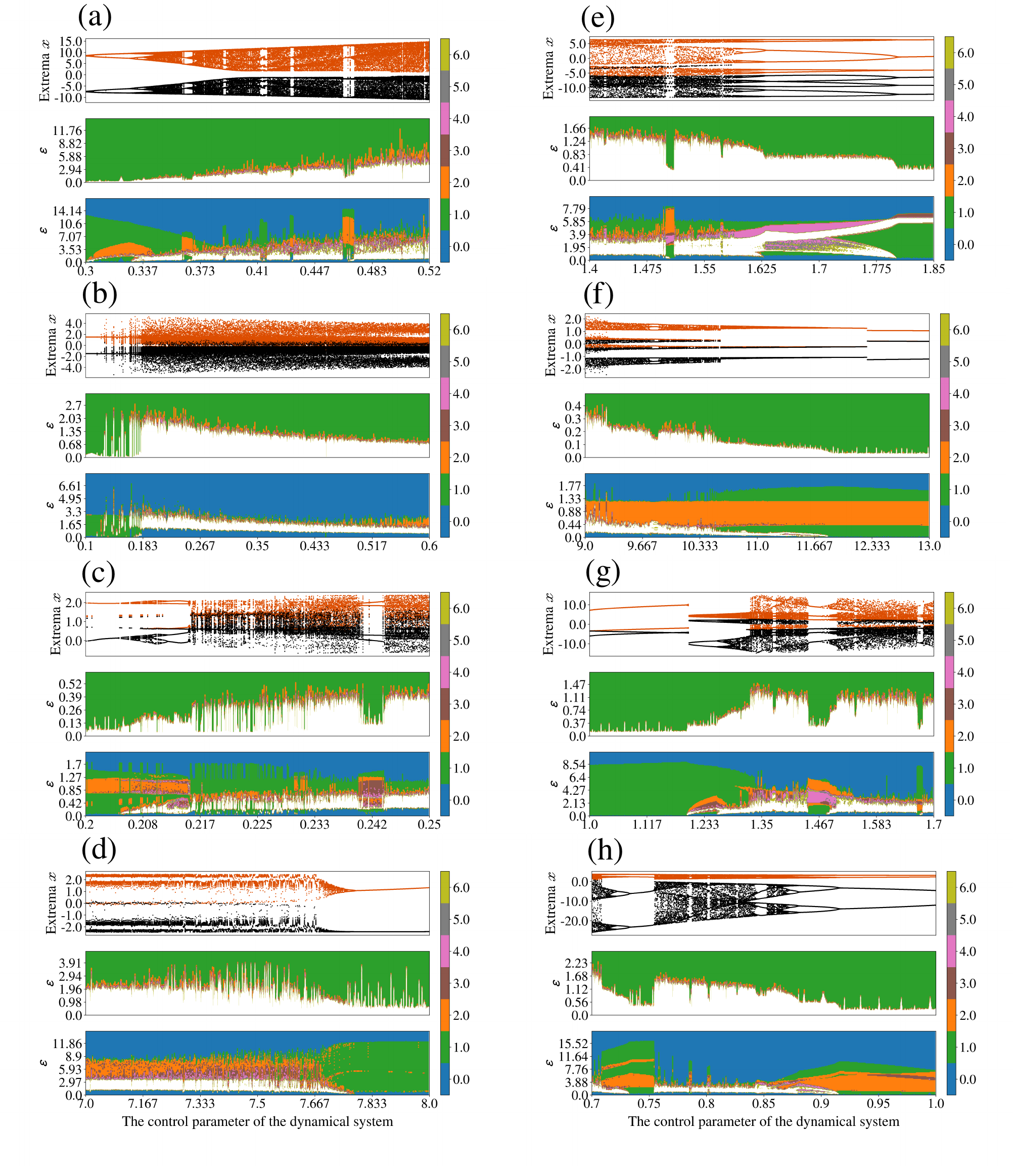}
        \caption{\revision{The bifurcation diagram and 0- and 1- dimensional CROCKER for each dynamical system in Appendix.~\ref{App:DynamicalSystem}. ((a)-Coupled Lorenz R\"{o}ssler, (b)- Complex Butterfly, (c)- Hadley circulation, (d)- Moore-Spiegel oscilator, (e)- Halvorsens cyclically symmetric attractor, (f)-Burke-Shaw attractor, (g)- Rucklidge attractor and (h)- WINDMI)}}
        \label{fig:allcrockers}
    \end{figure*}

\bibliography{article}%

\providecommand{\noopsort}[1]{}\providecommand{\singleletter}[1]{#1}%
\begin{thebibliography}{75}%
\makeatletter
\providecommand \@ifxundefined [1]{%
 \@ifx{#1\undefined}
}%
\providecommand \@ifnum [1]{%
 \ifnum #1\expandafter \@firstoftwo
 \else \expandafter \@secondoftwo
 \fi
}%
\providecommand \@ifx [1]{%
 \ifx #1\expandafter \@firstoftwo
 \else \expandafter \@secondoftwo
 \fi
}%
\providecommand \natexlab [1]{#1}%
\providecommand \enquote  [1]{``#1''}%
\providecommand \bibnamefont  [1]{#1}%
\providecommand \bibfnamefont [1]{#1}%
\providecommand \citenamefont [1]{#1}%
\providecommand \href@noop [0]{\@secondoftwo}%
\providecommand \href [0]{\begingroup \@sanitize@url \@href}%
\providecommand \@href[1]{\@@startlink{#1}\@@href}%
\providecommand \@@href[1]{\endgroup#1\@@endlink}%
\providecommand \@sanitize@url [0]{\catcode `\\12\catcode `\$12\catcode
  `\&12\catcode `\#12\catcode `\^12\catcode `\_12\catcode `\%12\relax}%
\providecommand \@@startlink[1]{}%
\providecommand \@@endlink[0]{}%
\providecommand \url  [0]{\begingroup\@sanitize@url \@url }%
\providecommand \@url [1]{\endgroup\@href {#1}{\urlprefix }}%
\providecommand \urlprefix  [0]{URL }%
\providecommand \Eprint [0]{\href }%
\providecommand \doibase [0]{http://dx.doi.org/}%
\providecommand \selectlanguage [0]{\@gobble}%
\providecommand \bibinfo  [0]{\@secondoftwo}%
\providecommand \bibfield  [0]{\@secondoftwo}%
\providecommand \translation [1]{[#1]}%
\providecommand \BibitemOpen [0]{}%
\providecommand \bibitemStop [0]{}%
\providecommand \bibitemNoStop [0]{.\EOS\space}%
\providecommand \EOS [0]{\spacefactor3000\relax}%
\providecommand \BibitemShut  [1]{\csname bibitem#1\endcsname}%
\let\auto@bib@innerbib\@empty
%</preamble>
\bibitem [{\citenamefont {Topaz}, \citenamefont {Ziegelmeier},\ and\
  \citenamefont {Halverson}(2015)}]{crocker-Lori2015}%
  \BibitemOpen
  \bibfield  {author} {\bibinfo {author} {\bibfnamefont {C.~M.}\ \bibnamefont
  {Topaz}}, \bibinfo {author} {\bibfnamefont {L.}~\bibnamefont {Ziegelmeier}},
  \ and\ \bibinfo {author} {\bibfnamefont {T.}~\bibnamefont {Halverson}},\
  }\bibfield  {title} {\enquote {\bibinfo {title} {Topological data analysis of
  biological aggregation models},}\ }\href {\doibase
  https://doi.org/10.1371/journal.pone.0126383} {\bibfield  {journal} {\bibinfo
   {journal} {PLOS ONE}\ }\textbf {\bibinfo {volume} {10}},\ \bibinfo {pages}
  {1--26} (\bibinfo {year} {2015})}\BibitemShut {NoStop}%
\bibitem [{\citenamefont {Robinson}(2014)}]{robinson2014topological}%
  \BibitemOpen
  \bibfield  {author} {\bibinfo {author} {\bibfnamefont {M.}~\bibnamefont
  {Robinson}},\ }\href {\doibase https://doi.org/10.1007/978-3-642-36104-3}
  {\emph {\bibinfo {title} {Topological signal processing}}},\ Vol.~\bibinfo
  {volume} {81}\ (\bibinfo  {publisher} {Springer},\ \bibinfo {year}
  {2014})\BibitemShut {NoStop}%
\bibitem [{\citenamefont {Robins}, \citenamefont {Meiss},\ and\ \citenamefont
  {Bradley}(2000)}]{Robins2000}%
  \BibitemOpen
  \bibfield  {author} {\bibinfo {author} {\bibfnamefont {V.}~\bibnamefont
  {Robins}}, \bibinfo {author} {\bibfnamefont {J.~D.}\ \bibnamefont {Meiss}}, \
  and\ \bibinfo {author} {\bibfnamefont {E.}~\bibnamefont {Bradley}},\
  }{\selectlanguage {English}\emph {\bibinfo {title} {Computational topology at
  multiple resolutions: foundations and applications to fractals and
  dynamics}}},\ \href {http://ezproxy.msu.edu/login} {Ph.D. thesis},\ \bibinfo
  {school} {University of Colorado} (\bibinfo {year} {2000})\BibitemShut
  {NoStop}%
\bibitem [{\citenamefont {Kaczynski}, \citenamefont {Mischaikow},\ and\
  \citenamefont {Mrozek}(2004)}]{Kaczynski2004}%
  \BibitemOpen
  \bibfield  {author} {\bibinfo {author} {\bibfnamefont {T.}~\bibnamefont
  {Kaczynski}}, \bibinfo {author} {\bibfnamefont {K.}~\bibnamefont
  {Mischaikow}}, \ and\ \bibinfo {author} {\bibfnamefont {M.}~\bibnamefont
  {Mrozek}},\ }\href
  {http://www.ebook.de/de/product/3677309/tomasz_kaczynski_konstantin_mischaikow_marian_mrozek_computational_homology.html}
  {\emph {\bibinfo {title} {Computational Homology}}}\ (\bibinfo  {publisher}
  {Springer},\ \bibinfo {year} {2004})\BibitemShut {NoStop}%
\bibitem [{\citenamefont {Khasawneh}\ and\ \citenamefont
  {Munch}(2017)}]{khasawneh2017utilizing}%
  \BibitemOpen
  \bibfield  {author} {\bibinfo {author} {\bibfnamefont {F.~A.}\ \bibnamefont
  {Khasawneh}}\ and\ \bibinfo {author} {\bibfnamefont {E.}~\bibnamefont
  {Munch}},\ }\bibfield  {title} {\enquote {\bibinfo {title} {Utilizing
  topological data analysis for studying signals of time-delay systems},}\ }in\
  \href {\doibase https://doi.org/10.1007/978-3-319-53426-8_7} {\emph {\bibinfo
  {booktitle} {Time Delay Systems}}}\ (\bibinfo  {publisher} {Springer},\
  \bibinfo {year} {2017})\ pp.\ \bibinfo {pages} {93--106}\BibitemShut
  {NoStop}%
\bibitem [{\citenamefont {Garland}, \citenamefont {Bradley},\ and\
  \citenamefont {Meiss}(6 11)}]{Garland2016}%
  \BibitemOpen
  \bibfield  {author} {\bibinfo {author} {\bibfnamefont {J.}~\bibnamefont
  {Garland}}, \bibinfo {author} {\bibfnamefont {E.}~\bibnamefont {Bradley}}, \
  and\ \bibinfo {author} {\bibfnamefont {J.~D.}\ \bibnamefont {Meiss}},\
  }\bibfield  {title} {\enquote {\bibinfo {title} {Exploring the topology of
  dynamical reconstructions},}\ }\href {\doibase 10.1016/j.physd.2016.03.006}
  {\bibfield  {journal} {\bibinfo  {journal} {Physica D: Nonlinear Phenomena}\
  }\textbf {\bibinfo {volume} {334}},\ \bibinfo {pages} {49--59} (\bibinfo
  {year} {2016-11})}\BibitemShut {NoStop}%
\bibitem [{\citenamefont {Mittal}\ and\ \citenamefont
  {Gupta}(2017)}]{mittal2017topological}%
  \BibitemOpen
  \bibfield  {author} {\bibinfo {author} {\bibfnamefont {K.}~\bibnamefont
  {Mittal}}\ and\ \bibinfo {author} {\bibfnamefont {S.}~\bibnamefont {Gupta}},\
  }\bibfield  {title} {\enquote {\bibinfo {title} {Topological characterization
  and early detection of bifurcations and chaos in complex systems using
  persistent homology},}\ }\href {\doibase https://doi.org/10.1063/1.4983840}
  {\bibfield  {journal} {\bibinfo  {journal} {Chaos: An Interdisciplinary
  Journal of Nonlinear Science}\ }\textbf {\bibinfo {volume} {27}},\ \bibinfo
  {pages} {051102} (\bibinfo {year} {2017})}\BibitemShut {NoStop}%
\bibitem [{\citenamefont {Dee~Algar}, \citenamefont {Corrêa},\ and\
  \citenamefont {Walker}(2021)}]{TACTS}%
  \BibitemOpen
  \bibfield  {author} {\bibinfo {author} {\bibfnamefont {S.}~\bibnamefont
  {Dee~Algar}}, \bibinfo {author} {\bibfnamefont {D.~C.}\ \bibnamefont
  {Corrêa}}, \ and\ \bibinfo {author} {\bibfnamefont {D.~M.}\ \bibnamefont
  {Walker}},\ }\bibfield  {title} {\enquote {\bibinfo {title} {On detecting
  dynamical regime change using a transformation cost metric between persistent
  homology diagrams},}\ }\href {\doibase 10.1063/5.0073247} {\bibfield
  {journal} {\bibinfo  {journal} {Chaos: An Interdisciplinary Journal of
  Nonlinear Science}\ }\textbf {\bibinfo {volume} {31}},\ \bibinfo {pages}
  {123117} (\bibinfo {year} {2021})}\BibitemShut {NoStop}%
\bibitem [{\citenamefont {Güzel}\ and\ \citenamefont
  {Kaygun}(2022)}]{Guzel2022stochastic}%
  \BibitemOpen
  \bibfield  {author} {\bibinfo {author} {\bibfnamefont {I.}~\bibnamefont
  {Güzel}}\ and\ \bibinfo {author} {\bibfnamefont {A.}~\bibnamefont
  {Kaygun}},\ }\bibfield  {title} {\enquote {\bibinfo {title} {Classification
  of stochastic processes with topological data analysis},}\ }in\ \href
  {\doibase 10.48550/arXiv.2206.03973} {\emph {\bibinfo {booktitle}
  {Proceedings of BAŞARIM 2022 - 7th High-Performance Computing Conference}}}\
  (\bibinfo {organization} {BAŞARIM},\ \bibinfo {year} {2022})\ \Eprint
  {http://arxiv.org/abs/2206.03973} {arXiv:2206.03973} \BibitemShut {NoStop}%
\bibitem [{\citenamefont {Karan}\ and\ \citenamefont
  {Kaygun}(2021)}]{Karan2021}%
  \BibitemOpen
  \bibfield  {author} {\bibinfo {author} {\bibfnamefont {A.}~\bibnamefont
  {Karan}}\ and\ \bibinfo {author} {\bibfnamefont {A.}~\bibnamefont {Kaygun}},\
  }\bibfield  {title} {\enquote {\bibinfo {title} {Time series classification
  via topological data analysis},}\ }\href {\doibase
  10.1016/j.eswa.2021.115326} {\bibfield  {journal} {\bibinfo  {journal}
  {Expert Systems with Applications}\ }\textbf {\bibinfo {volume} {183}},\
  \bibinfo {pages} {115326} (\bibinfo {year} {2021})}\BibitemShut {NoStop}%
\bibitem [{\citenamefont {Maleti{\'c}}, \citenamefont {Zhao},\ and\
  \citenamefont {Rajkovi{\'c}}(2016)}]{maletic2016persistent}%
  \BibitemOpen
  \bibfield  {author} {\bibinfo {author} {\bibfnamefont {S.}~\bibnamefont
  {Maleti{\'c}}}, \bibinfo {author} {\bibfnamefont {Y.}~\bibnamefont {Zhao}}, \
  and\ \bibinfo {author} {\bibfnamefont {M.}~\bibnamefont {Rajkovi{\'c}}},\
  }\bibfield  {title} {\enquote {\bibinfo {title} {Persistent topological
  features of dynamical systems},}\ }\href {\doibase
  https://doi.org/10.1063/1.4949472} {\bibfield  {journal} {\bibinfo  {journal}
  {Chaos: An Interdisciplinary Journal of Nonlinear Science}\ }\textbf
  {\bibinfo {volume} {26}},\ \bibinfo {pages} {053105} (\bibinfo {year}
  {2016})}\BibitemShut {NoStop}%
\bibitem [{\citenamefont {Marchese}\ and\ \citenamefont
  {Maroulas}(2017)}]{Marchese2017}%
  \BibitemOpen
  \bibfield  {author} {\bibinfo {author} {\bibfnamefont {A.}~\bibnamefont
  {Marchese}}\ and\ \bibinfo {author} {\bibfnamefont {V.}~\bibnamefont
  {Maroulas}},\ }\bibfield  {title} {\enquote {\bibinfo {title} {Signal
  classification with a point process distance on the space of persistence
  diagrams},}\ }\href {\doibase 10.1007/s11634-017-0294-x} {\bibfield
  {journal} {\bibinfo  {journal} {Advances in Data Analysis and
  Classification}\ } (\bibinfo {year} {2017}),\
  10.1007/s11634-017-0294-x}\BibitemShut {NoStop}%
\bibitem [{\citenamefont {Kim}, \citenamefont {Kim},\ and\ \citenamefont
  {Rinaldo}(2018)}]{Kim2018c}%
  \BibitemOpen
  \bibfield  {author} {\bibinfo {author} {\bibfnamefont {K.}~\bibnamefont
  {Kim}}, \bibinfo {author} {\bibfnamefont {J.}~\bibnamefont {Kim}}, \ and\
  \bibinfo {author} {\bibfnamefont {A.}~\bibnamefont {Rinaldo}},\ }\bibfield
  {title} {\enquote {\bibinfo {title} {Time series featurization via
  topological data analysis},}\ }\href@noop {} {\bibfield  {journal} {\bibinfo
  {journal} {arXiv preprint arXiv:1812.02987}\ } (\bibinfo {year} {2018})},\
  \Eprint {http://arxiv.org/abs/1812.02987} {arXiv:1812.02987 [cs.CG]}
  \BibitemShut {NoStop}%
\bibitem [{\citenamefont {Seversky}, \citenamefont {Davis},\ and\ \citenamefont
  {Berger}(2016)}]{Seversky2016}%
  \BibitemOpen
  \bibfield  {author} {\bibinfo {author} {\bibfnamefont {L.~M.}\ \bibnamefont
  {Seversky}}, \bibinfo {author} {\bibfnamefont {S.}~\bibnamefont {Davis}}, \
  and\ \bibinfo {author} {\bibfnamefont {M.}~\bibnamefont {Berger}},\
  }\bibfield  {title} {\enquote {\bibinfo {title} {On time-series topological
  data analysis: New data and opportunities},}\ }in\ \href {\doibase
  https://doi.org/10.1109/CVPRW.2016.131} {\emph {\bibinfo {booktitle} {The
  IEEE Conference on Computer Vision and Pattern Recognition (CVPR)
  Workshops}}}\ (\bibinfo {year} {2016})\BibitemShut {NoStop}%
\bibitem [{\citenamefont {Venkataraman}, \citenamefont {Ramamurthy},\ and\
  \citenamefont {Turaga}(2016)}]{Venkataraman2016}%
  \BibitemOpen
  \bibfield  {author} {\bibinfo {author} {\bibfnamefont {V.}~\bibnamefont
  {Venkataraman}}, \bibinfo {author} {\bibfnamefont {K.~N.}\ \bibnamefont
  {Ramamurthy}}, \ and\ \bibinfo {author} {\bibfnamefont {P.}~\bibnamefont
  {Turaga}},\ }\bibfield  {title} {\enquote {\bibinfo {title} {Persistent
  homology of attractors for action recognition},}\ }in\ \href {\doibase
  10.1109/icip.2016.7533141} {\emph {\bibinfo {booktitle} {2016 {IEEE}
  International Conference on Image Processing ({ICIP})}}}\ (\bibinfo
  {publisher} {{IEEE}},\ \bibinfo {year} {2016})\BibitemShut {NoStop}%
\bibitem [{\citenamefont {Perea}\ and\ \citenamefont
  {Harer}(2015)}]{Perea2014}%
  \BibitemOpen
  \bibfield  {author} {\bibinfo {author} {\bibfnamefont {J.~A.}\ \bibnamefont
  {Perea}}\ and\ \bibinfo {author} {\bibfnamefont {J.}~\bibnamefont {Harer}},\
  }\bibfield  {title} {{\selectlanguage {English}\enquote {\bibinfo {title}
  {Sliding windows and persistence: An application of topological methods to
  signal analysis},}\ }}\href {\doibase 10.1007/s10208-014-9206-z} {\bibfield
  {journal} {\bibinfo  {journal} {Foundations of Computational Mathematics}\ ,\
  \bibinfo {pages} {1--40}} (\bibinfo {year} {2015})}\BibitemShut {NoStop}%
\bibitem [{\citenamefont {Perea}(2019)}]{Perea2019a}%
  \BibitemOpen
  \bibfield  {author} {\bibinfo {author} {\bibfnamefont {J.~A.}\ \bibnamefont
  {Perea}},\ }\bibfield  {title} {\enquote {\bibinfo {title} {Topological time
  series analysis},}\ }\href {\doibase 10.1090/noti1869} {\bibfield  {journal}
  {\bibinfo  {journal} {Notices of the American Mathematical Society}\ }\textbf
  {\bibinfo {volume} {66}},\ \bibinfo {pages} {686--694} (\bibinfo {year}
  {2019})}\BibitemShut {NoStop}%
\bibitem [{\citenamefont {Berwald}, \citenamefont {Gidea},\ and\ \citenamefont
  {Vejdemo-Johansson}(2014)}]{Berwald2014a}%
  \BibitemOpen
  \bibfield  {author} {\bibinfo {author} {\bibfnamefont {J.~J.}\ \bibnamefont
  {Berwald}}, \bibinfo {author} {\bibfnamefont {M.}~\bibnamefont {Gidea}}, \
  and\ \bibinfo {author} {\bibfnamefont {M.}~\bibnamefont
  {Vejdemo-Johansson}},\ }\bibfield  {title} {\enquote {\bibinfo {title}
  {Automatic recognition and tagging of topologically different regimes in
  dynamical systems},}\ }\href
  {https://www.lhscientificpublishing.com/Journals/articles/DOI-10.5890-DNC.2014.12.004.aspx}
  {\bibfield  {journal} {\bibinfo  {journal} {Discontinuity, Nonlinearity, and
  Complexity}\ }\textbf {\bibinfo {volume} {3}},\ \bibinfo {pages} {413--426}
  (\bibinfo {year} {2014})}\BibitemShut {NoStop}%
\bibitem [{\citenamefont {Tan}\ \emph {et~al.}(2021)\citenamefont {Tan},
  \citenamefont {Corr{\^{e}}a}, \citenamefont {Stemler},\ and\ \citenamefont
  {Small}}]{Tan2021}%
  \BibitemOpen
  \bibfield  {author} {\bibinfo {author} {\bibfnamefont {E.}~\bibnamefont
  {Tan}}, \bibinfo {author} {\bibfnamefont {D.}~\bibnamefont {Corr{\^{e}}a}},
  \bibinfo {author} {\bibfnamefont {T.}~\bibnamefont {Stemler}}, \ and\
  \bibinfo {author} {\bibfnamefont {M.}~\bibnamefont {Small}},\ }\bibfield
  {title} {\enquote {\bibinfo {title} {Grading your models: Assessing dynamics
  learning of models using persistent homology},}\ }\href {\doibase
  10.1063/5.0073722} {\bibfield  {journal} {\bibinfo  {journal} {Chaos: An
  Interdisciplinary Journal of Nonlinear Science}\ }\textbf {\bibinfo {volume}
  {31}},\ \bibinfo {pages} {123109} (\bibinfo {year} {2021})}\BibitemShut
  {NoStop}%
\bibitem [{\citenamefont {Bubenik}\ and\ \citenamefont
  {Wagner}(2020)}]{Bubenik2020a}%
  \BibitemOpen
  \bibfield  {author} {\bibinfo {author} {\bibfnamefont {P.}~\bibnamefont
  {Bubenik}}\ and\ \bibinfo {author} {\bibfnamefont {A.}~\bibnamefont
  {Wagner}},\ }\bibfield  {title} {\enquote {\bibinfo {title} {Embeddings of
  persistence diagrams into hilbert spaces},}\ }\href {\doibase
  10.1007/s41468-020-00056-w} {\bibfield  {journal} {\bibinfo  {journal}
  {Journal of Applied and Computational Topology}\ }\textbf {\bibinfo {volume}
  {4}},\ \bibinfo {pages} {339--351} (\bibinfo {year} {2020})}\BibitemShut
  {NoStop}%
\bibitem [{\citenamefont {Yesilli}\ \emph {et~al.}(2019)\citenamefont
  {Yesilli}, \citenamefont {Tymochko}, \citenamefont {Khasawneh},\ and\
  \citenamefont {Munch}}]{Yesilli2019b}%
  \BibitemOpen
  \bibfield  {author} {\bibinfo {author} {\bibfnamefont {M.~C.}\ \bibnamefont
  {Yesilli}}, \bibinfo {author} {\bibfnamefont {S.}~\bibnamefont {Tymochko}},
  \bibinfo {author} {\bibfnamefont {F.~A.}\ \bibnamefont {Khasawneh}}, \ and\
  \bibinfo {author} {\bibfnamefont {E.}~\bibnamefont {Munch}},\ }\bibfield
  {title} {\enquote {\bibinfo {title} {Chatter diagnosis in milling using
  supervised learning and topological features vector},}\ }\href {\doibase
  10.1109/ICMLA.2019.00200} {\bibfield  {journal} {\bibinfo  {journal} {2019
  18th IEEE International Conference On Machine Learning And Applications
  (ICMLA)}\ } (\bibinfo {year} {2019}),\ 10.1109/ICMLA.2019.00200},\ \bibinfo
  {note} {special session acceptance rate: 43\%},\ \Eprint
  {http://arxiv.org/abs/1910.12359} {arXiv:1910.12359} \BibitemShut {NoStop}%
\bibitem [{\citenamefont {Khasawneh}\ and\ \citenamefont
  {Munch}(2016)}]{Khasawneh2015}%
  \BibitemOpen
  \bibfield  {author} {\bibinfo {author} {\bibfnamefont {F.~A.}\ \bibnamefont
  {Khasawneh}}\ and\ \bibinfo {author} {\bibfnamefont {E.}~\bibnamefont
  {Munch}},\ }\bibfield  {title} {\enquote {\bibinfo {title} {Chatter detection
  in turning using persistent homology},}\ }\href {\doibase
  http://dx.doi.org/10.1016/j.ymssp.2015.09.046} {\bibfield  {journal}
  {\bibinfo  {journal} {Mechanical Systems and Signal Processing}\ }\textbf
  {\bibinfo {volume} {70-71}},\ \bibinfo {pages} {527--541} (\bibinfo {year}
  {2016})}\BibitemShut {NoStop}%
\bibitem [{\citenamefont {Khasawneh}\ and\ \citenamefont
  {Munch}(2018)}]{Khasawneh2018}%
  \BibitemOpen
  \bibfield  {author} {\bibinfo {author} {\bibfnamefont {F.~A.}\ \bibnamefont
  {Khasawneh}}\ and\ \bibinfo {author} {\bibfnamefont {E.}~\bibnamefont
  {Munch}},\ }\bibfield  {title} {\enquote {\bibinfo {title} {Topological data
  analysis for true step detection in periodic piecewise constant signals},}\
  }\href {\doibase 10.1098/rspa.2018.0027} {\bibfield  {journal} {\bibinfo
  {journal} {Proceedings of the Royal Society A: Mathematical, Physical and
  Engineering Sciences}\ }\textbf {\bibinfo {volume} {474}},\ \bibinfo {pages}
  {20180027} (\bibinfo {year} {2018})}\BibitemShut {NoStop}%
\bibitem [{\citenamefont {Gidea}(2017)}]{Gidea2017a}%
  \BibitemOpen
  \bibfield  {author} {\bibinfo {author} {\bibfnamefont {M.}~\bibnamefont
  {Gidea}},\ }\enquote {\bibinfo {title} {Topological data analysis of critical
  transitions in financial networks},}\ in\ \href {\doibase
  10.1007/978-3-319-55471-6_5} {\emph {\bibinfo {booktitle} {3rd International
  Winter School and Conference on Network Science : NetSci-X 2017}}}\ (\bibinfo
   {publisher} {Springer International Publishing},\ \bibinfo {address}
  {Cham},\ \bibinfo {year} {2017})\ pp.\ \bibinfo {pages} {47--59}\BibitemShut
  {NoStop}%
\bibitem [{\citenamefont {Gidea}\ and\ \citenamefont
  {Katz}(2018)}]{Gidea2018a}%
  \BibitemOpen
  \bibfield  {author} {\bibinfo {author} {\bibfnamefont {M.}~\bibnamefont
  {Gidea}}\ and\ \bibinfo {author} {\bibfnamefont {Y.}~\bibnamefont {Katz}},\
  }\bibfield  {title} {\enquote {\bibinfo {title} {Topological data analysis of
  financial time series: Landscapes of crashes},}\ }\href {\doibase
  10.1016/j.physa.2017.09.028} {\bibfield  {journal} {\bibinfo  {journal}
  {Physica A: Statistical Mechanics and its Applications}\ }\textbf {\bibinfo
  {volume} {491}},\ \bibinfo {pages} {820--834} (\bibinfo {year} {2018})},\
  \Eprint {http://arxiv.org/abs/1703.04385v2} {1703.04385v2} \BibitemShut
  {NoStop}%
\bibitem [{\citenamefont {Gidea}\ \emph {et~al.}(2020)\citenamefont {Gidea},
  \citenamefont {Goldsmith}, \citenamefont {Katz}, \citenamefont {Roldan},\
  and\ \citenamefont {Shmalo}}]{Gidea2020}%
  \BibitemOpen
  \bibfield  {author} {\bibinfo {author} {\bibfnamefont {M.}~\bibnamefont
  {Gidea}}, \bibinfo {author} {\bibfnamefont {D.}~\bibnamefont {Goldsmith}},
  \bibinfo {author} {\bibfnamefont {Y.}~\bibnamefont {Katz}}, \bibinfo {author}
  {\bibfnamefont {P.}~\bibnamefont {Roldan}}, \ and\ \bibinfo {author}
  {\bibfnamefont {Y.}~\bibnamefont {Shmalo}},\ }\bibfield  {title} {\enquote
  {\bibinfo {title} {Topological recognition of critical transitions in time
  series of cryptocurrencies},}\ }\href {\doibase 10.1016/j.physa.2019.123843}
  {\bibfield  {journal} {\bibinfo  {journal} {Physica A: Statistical Mechanics
  and its Applications}\ }\textbf {\bibinfo {volume} {548}},\ \bibinfo {pages}
  {123843} (\bibinfo {year} {2020})},\ \Eprint
  {http://arxiv.org/abs/1809.00695v1} {1809.00695v1} \BibitemShut {NoStop}%
\bibitem [{\citenamefont {Rivera-Castro}, \citenamefont {Pilyugina},\ and\
  \citenamefont {Burnaev}(2019)}]{RiveraCastro2019}%
  \BibitemOpen
  \bibfield  {author} {\bibinfo {author} {\bibfnamefont {R.}~\bibnamefont
  {Rivera-Castro}}, \bibinfo {author} {\bibfnamefont {P.}~\bibnamefont
  {Pilyugina}}, \ and\ \bibinfo {author} {\bibfnamefont {E.}~\bibnamefont
  {Burnaev}},\ }\bibfield  {title} {\enquote {\bibinfo {title} {Topological
  data analysis for portfolio management of cryptocurrencies},}\ }in\ \href
  {\doibase 10.1109/icdmw.2019.00044} {\emph {\bibinfo {booktitle} {2019
  International Conference on Data Mining Workshops ({ICDMW})}}}\ (\bibinfo
  {publisher} {{IEEE}},\ \bibinfo {year} {2019})\BibitemShut {NoStop}%
\bibitem [{\citenamefont {Majumdar}\ and\ \citenamefont
  {Laha}(2020)}]{Majumdar2020}%
  \BibitemOpen
  \bibfield  {author} {\bibinfo {author} {\bibfnamefont {S.}~\bibnamefont
  {Majumdar}}\ and\ \bibinfo {author} {\bibfnamefont {A.~K.}\ \bibnamefont
  {Laha}},\ }\bibfield  {title} {\enquote {\bibinfo {title} {Clustering and
  classification of time series using topological data analysis with
  applications to finance},}\ }\href {\doibase 10.1016/j.eswa.2020.113868}
  {\bibfield  {journal} {\bibinfo  {journal} {Expert Systems with
  Applications}\ }\textbf {\bibinfo {volume} {162}},\ \bibinfo {pages} {113868}
  (\bibinfo {year} {2020})}\BibitemShut {NoStop}%
\bibitem [{\citenamefont {Ignacio}\ \emph {et~al.}(2019)\citenamefont
  {Ignacio}, \citenamefont {Dunstan}, \citenamefont {Escobar}, \citenamefont
  {Trujillo},\ and\ \citenamefont {Uminsky}}]{Ignacio2019a}%
  \BibitemOpen
  \bibfield  {author} {\bibinfo {author} {\bibfnamefont {P.~S.}\ \bibnamefont
  {Ignacio}}, \bibinfo {author} {\bibfnamefont {C.}~\bibnamefont {Dunstan}},
  \bibinfo {author} {\bibfnamefont {E.}~\bibnamefont {Escobar}}, \bibinfo
  {author} {\bibfnamefont {L.}~\bibnamefont {Trujillo}}, \ and\ \bibinfo
  {author} {\bibfnamefont {D.}~\bibnamefont {Uminsky}},\ }\bibfield  {title}
  {\enquote {\bibinfo {title} {Classification of single-lead
  electrocardiograms: {TDA} informed machine learning},}\ }in\ \href {\doibase
  10.1109/icmla.2019.00204} {\emph {\bibinfo {booktitle} {2019 18th {IEEE}
  International Conference On Machine Learning And Applications ({ICMLA})}}}\
  (\bibinfo  {publisher} {{IEEE}},\ \bibinfo {year} {2019})\BibitemShut
  {NoStop}%
\bibitem [{\citenamefont {Stolz}, \citenamefont {Harrington},\ and\
  \citenamefont {Porter}(2017)}]{Stolz2017}%
  \BibitemOpen
  \bibfield  {author} {\bibinfo {author} {\bibfnamefont {B.~J.}\ \bibnamefont
  {Stolz}}, \bibinfo {author} {\bibfnamefont {H.~A.}\ \bibnamefont
  {Harrington}}, \ and\ \bibinfo {author} {\bibfnamefont {M.~A.}\ \bibnamefont
  {Porter}},\ }\bibfield  {title} {\enquote {\bibinfo {title} {Persistent
  homology of time-dependent functional networks constructed from coupled time
  series},}\ }\href {\doibase 10.1063/1.4978997} {\bibfield  {journal}
  {\bibinfo  {journal} {Chaos: An Interdisciplinary Journal of Nonlinear
  Science}\ }\textbf {\bibinfo {volume} {27}},\ \bibinfo {pages} {047410}
  (\bibinfo {year} {2017})}\BibitemShut {NoStop}%
\bibitem [{\citenamefont {Majumder}\ \emph {et~al.}(2020)\citenamefont
  {Majumder}, \citenamefont {Apicella}, \citenamefont {Muratori},\ and\
  \citenamefont {Das}}]{Majumder2020}%
  \BibitemOpen
  \bibfield  {author} {\bibinfo {author} {\bibfnamefont {S.}~\bibnamefont
  {Majumder}}, \bibinfo {author} {\bibfnamefont {F.}~\bibnamefont {Apicella}},
  \bibinfo {author} {\bibfnamefont {F.}~\bibnamefont {Muratori}}, \ and\
  \bibinfo {author} {\bibfnamefont {K.}~\bibnamefont {Das}},\ }\bibfield
  {title} {\enquote {\bibinfo {title} {Detecting autism spectrum disorder using
  topological data analysis},}\ }in\ \href {\doibase
  10.1109/icassp40776.2020.9054747} {\emph {\bibinfo {booktitle} {{ICASSP} 2020
  - 2020 {IEEE} International Conference on Acoustics, Speech and Signal
  Processing ({ICASSP})}}}\ (\bibinfo  {publisher} {{IEEE}},\ \bibinfo {year}
  {2020})\BibitemShut {NoStop}%
\bibitem [{\citenamefont {Chung}\ \emph {et~al.}(2021)\citenamefont {Chung},
  \citenamefont {Hu}, \citenamefont {Lo},\ and\ \citenamefont
  {Wu}}]{Chung2021}%
  \BibitemOpen
  \bibfield  {author} {\bibinfo {author} {\bibfnamefont {Y.-M.}\ \bibnamefont
  {Chung}}, \bibinfo {author} {\bibfnamefont {C.-S.}\ \bibnamefont {Hu}},
  \bibinfo {author} {\bibfnamefont {Y.-L.}\ \bibnamefont {Lo}}, \ and\ \bibinfo
  {author} {\bibfnamefont {H.-T.}\ \bibnamefont {Wu}},\ }\bibfield  {title}
  {\enquote {\bibinfo {title} {A persistent homology approach to heart rate
  variability analysis with an application to sleep-wake classification},}\
  }\href {\doibase 10.3389/fphys.2021.637684} {\bibfield  {journal} {\bibinfo
  {journal} {Frontiers in Physiology}\ }\textbf {\bibinfo {volume} {12}}
  (\bibinfo {year} {2021}),\ 10.3389/fphys.2021.637684}\BibitemShut {NoStop}%
\bibitem [{\citenamefont {Emrani}, \citenamefont {Gentimis},\ and\
  \citenamefont {Krim}(2014)}]{Emrani2014}%
  \BibitemOpen
  \bibfield  {author} {\bibinfo {author} {\bibfnamefont {S.}~\bibnamefont
  {Emrani}}, \bibinfo {author} {\bibfnamefont {T.}~\bibnamefont {Gentimis}}, \
  and\ \bibinfo {author} {\bibfnamefont {H.}~\bibnamefont {Krim}},\ }\bibfield
  {title} {\enquote {\bibinfo {title} {Persistent homology of delay embeddings
  and its application to wheeze detection},}\ }\href {\doibase
  10.1109/LSP.2014.2305700} {\bibfield  {journal} {\bibinfo  {journal} {Signal
  Processing Letters, IEEE}\ }\textbf {\bibinfo {volume} {21}},\ \bibinfo
  {pages} {459--463} (\bibinfo {year} {2014})}\BibitemShut {NoStop}%
\bibitem [{\citenamefont {Tymochko}, \citenamefont {Munch},\ and\ \citenamefont
  {Khasawneh}(2020)}]{tymochko2020using}%
  \BibitemOpen
  \bibfield  {author} {\bibinfo {author} {\bibfnamefont {S.}~\bibnamefont
  {Tymochko}}, \bibinfo {author} {\bibfnamefont {E.}~\bibnamefont {Munch}}, \
  and\ \bibinfo {author} {\bibfnamefont {F.~A.}\ \bibnamefont {Khasawneh}},\
  }\bibfield  {title} {\enquote {\bibinfo {title} {Using zigzag persistent
  homology to detect hopf bifurcations in dynamical systems},}\ }\href
  {\doibase https://doi.org/10.3390/a13110278} {\bibfield  {journal} {\bibinfo
  {journal} {Algorithms}\ }\textbf {\bibinfo {volume} {13}},\ \bibinfo {pages}
  {278} (\bibinfo {year} {2020})}\BibitemShut {NoStop}%
\bibitem [{\citenamefont {Myers}, \citenamefont {Munch},\ and\ \citenamefont
  {Khasawneh}(2019)}]{myers2019persistent}%
  \BibitemOpen
  \bibfield  {author} {\bibinfo {author} {\bibfnamefont {A.}~\bibnamefont
  {Myers}}, \bibinfo {author} {\bibfnamefont {E.}~\bibnamefont {Munch}}, \ and\
  \bibinfo {author} {\bibfnamefont {F.~A.}\ \bibnamefont {Khasawneh}},\
  }\bibfield  {title} {\enquote {\bibinfo {title} {Persistent homology of
  complex networks for dynamic state detection},}\ }\href {\doibase
  https://doi.org/10.1103/PhysRevE.100.022314} {\bibfield  {journal} {\bibinfo
  {journal} {Physical Review E}\ }\textbf {\bibinfo {volume} {100}},\ \bibinfo
  {pages} {022314} (\bibinfo {year} {2019})}\BibitemShut {NoStop}%
\bibitem [{\citenamefont {Myers}, \citenamefont {Khasawneh},\ and\
  \citenamefont {Munch}(2022)}]{Myers2022}%
  \BibitemOpen
  \bibfield  {author} {\bibinfo {author} {\bibfnamefont {A.}~\bibnamefont
  {Myers}}, \bibinfo {author} {\bibfnamefont {F.~A.}\ \bibnamefont
  {Khasawneh}}, \ and\ \bibinfo {author} {\bibfnamefont {E.}~\bibnamefont
  {Munch}},\ }\bibfield  {title} {\enquote {\bibinfo {title} {Topological
  signal processing using the weighted ordinal partition network},}\
  }\href@noop {} {\bibfield  {journal} {\bibinfo  {journal} {arXiv preprint
  arXiv:2205.08349}\ } (\bibinfo {year} {2022})},\ \Eprint
  {http://arxiv.org/abs/2205.08349} {2205.08349} \BibitemShut {NoStop}%
\bibitem [{\citenamefont {Bauer}\ \emph {et~al.}(2020)\citenamefont {Bauer},
  \citenamefont {Hien}, \citenamefont {Junge}, \citenamefont {Mischaikow},\
  and\ \citenamefont {Snijders}}]{Bauer2020b}%
  \BibitemOpen
  \bibfield  {author} {\bibinfo {author} {\bibfnamefont {U.}~\bibnamefont
  {Bauer}}, \bibinfo {author} {\bibfnamefont {D.}~\bibnamefont {Hien}},
  \bibinfo {author} {\bibfnamefont {O.}~\bibnamefont {Junge}}, \bibinfo
  {author} {\bibfnamefont {K.}~\bibnamefont {Mischaikow}}, \ and\ \bibinfo
  {author} {\bibfnamefont {M.}~\bibnamefont {Snijders}},\ }\bibfield  {title}
  {\enquote {\bibinfo {title} {Combinatorial models of global dynamics:
  learning cycling motion from data},}\ }\href {\doibase
  10.1088/1361-6544/abe834} {\bibfield  {journal} {\bibinfo  {journal} {arXiv
  preprint arXiv:2001.07066}\ } (\bibinfo {year} {2020}),\
  10.1088/1361-6544/abe834},\ \Eprint {http://arxiv.org/abs/2001.07066}
  {arXiv:2001.07066 [math.DS]} \BibitemShut {NoStop}%
\bibitem [{\citenamefont {Tempelman}\ and\ \citenamefont
  {Khasawneh}(2020)}]{Tempelman2020}%
  \BibitemOpen
  \bibfield  {author} {\bibinfo {author} {\bibfnamefont {J.~R.}\ \bibnamefont
  {Tempelman}}\ and\ \bibinfo {author} {\bibfnamefont {F.~A.}\ \bibnamefont
  {Khasawneh}},\ }\bibfield  {title} {\enquote {\bibinfo {title} {A look into
  chaos detection through topological data analysis},}\ }\href {\doibase
  10.1016/j.physd.2020.132446} {\bibfield  {journal} {\bibinfo  {journal}
  {Physica D: Nonlinear Phenomena}\ }\textbf {\bibinfo {volume} {406}},\
  \bibinfo {pages} {132446} (\bibinfo {year} {2020})}\BibitemShut {NoStop}%
\bibitem [{\citenamefont {Myers}, \citenamefont {Khasawneh},\ and\
  \citenamefont {Munch}(2 05)}]{Myers2022a}%
  \BibitemOpen
  \bibfield  {author} {\bibinfo {author} {\bibfnamefont {A.}~\bibnamefont
  {Myers}}, \bibinfo {author} {\bibfnamefont {F.}~\bibnamefont {Khasawneh}}, \
  and\ \bibinfo {author} {\bibfnamefont {E.}~\bibnamefont {Munch}},\ }\bibfield
   {title} {\enquote {\bibinfo {title} {Temporal network analysis using zigzag
  persistence},}\ }\href@noop {} {\bibfield  {journal} {\bibinfo  {journal}
  {arXiv preprint arXiv:2205.11338}\ } (\bibinfo {year} {2022-05})},\ \Eprint
  {http://arxiv.org/abs/2205.11338} {2205.11338} \BibitemShut {NoStop}%
\bibitem [{\citenamefont {Tralie}(2016)}]{Tralie2016a}%
  \BibitemOpen
  \bibfield  {author} {\bibinfo {author} {\bibfnamefont {C.}~\bibnamefont
  {Tralie}},\ }\bibfield  {title} {\enquote {\bibinfo {title}
  {{High-Dimensional Geometry of Sliding Window Embeddings of Periodic
  Videos}},}\ }in\ \href {\doibase 10.4230/LIPIcs.SoCG.2016.71} {\emph
  {\bibinfo {booktitle} {32nd International Symposium on Computational Geometry
  (SoCG 2016)}}},\ \bibinfo {series} {Leibniz International Proceedings in
  Informatics (LIPIcs)}, Vol.~\bibinfo {volume} {51},\ \bibinfo {editor}
  {edited by\ \bibinfo {editor} {\bibfnamefont {S.}~\bibnamefont {Fekete}}\
  and\ \bibinfo {editor} {\bibfnamefont {A.}~\bibnamefont {Lubiw}}}\ (\bibinfo
  {publisher} {Schloss Dagstuhl--Leibniz-Zentrum fuer Informatik},\ \bibinfo
  {address} {Dagstuhl, Germany},\ \bibinfo {year} {2016})\ pp.\ \bibinfo
  {pages} {71:1--71:5}\BibitemShut {NoStop}%
\bibitem [{\citenamefont {Tymochko}\ \emph {et~al.}(2020)\citenamefont
  {Tymochko}, \citenamefont {Munch}, \citenamefont {Dunion}, \citenamefont
  {Corbosiero},\ and\ \citenamefont {Torn}}]{Tymochko2020}%
  \BibitemOpen
  \bibfield  {author} {\bibinfo {author} {\bibfnamefont {S.}~\bibnamefont
  {Tymochko}}, \bibinfo {author} {\bibfnamefont {E.}~\bibnamefont {Munch}},
  \bibinfo {author} {\bibfnamefont {J.}~\bibnamefont {Dunion}}, \bibinfo
  {author} {\bibfnamefont {K.}~\bibnamefont {Corbosiero}}, \ and\ \bibinfo
  {author} {\bibfnamefont {R.}~\bibnamefont {Torn}},\ }\bibfield  {title}
  {\enquote {\bibinfo {title} {Using persistent homology to quantify a diurnal
  cycle in hurricane felix},}\ }\href {\doibase 10.1016/j.patrec.2020.02.022}
  {\bibfield  {journal} {\bibinfo  {journal} {Pattern Recognition Letters}\
  }\textbf {\bibinfo {volume} {133}},\ \bibinfo {pages} {137--143} (\bibinfo
  {year} {2020})}\BibitemShut {NoStop}%
\bibitem [{\citenamefont {Tralie}\ and\ \citenamefont
  {Perea}(2018)}]{Tralie2017}%
  \BibitemOpen
  \bibfield  {author} {\bibinfo {author} {\bibfnamefont {C.~J.}\ \bibnamefont
  {Tralie}}\ and\ \bibinfo {author} {\bibfnamefont {J.~A.}\ \bibnamefont
  {Perea}},\ }\bibfield  {title} {\enquote {\bibinfo {title} {(quasi)
  periodicity quantification in video data, using topology},}\ }\href {\doibase
  https://doi.org/10.1137/17M1150736} {\bibfield  {journal} {\bibinfo
  {journal} {SIAM Journal on Imaging Sciences}\ }\textbf {\bibinfo {volume}
  {11}},\ \bibinfo {pages} {1049--1077} (\bibinfo {year} {2018})}\BibitemShut
  {NoStop}%
\bibitem [{\citenamefont {Levanger}\ \emph {et~al.}(2019)\citenamefont
  {Levanger}, \citenamefont {Xu}, \citenamefont {Cyranka}, \citenamefont
  {Schatz}, \citenamefont {Mischaikow},\ and\ \citenamefont
  {Paul}}]{Levanger2019}%
  \BibitemOpen
  \bibfield  {author} {\bibinfo {author} {\bibfnamefont {R.}~\bibnamefont
  {Levanger}}, \bibinfo {author} {\bibfnamefont {M.}~\bibnamefont {Xu}},
  \bibinfo {author} {\bibfnamefont {J.}~\bibnamefont {Cyranka}}, \bibinfo
  {author} {\bibfnamefont {M.~F.}\ \bibnamefont {Schatz}}, \bibinfo {author}
  {\bibfnamefont {K.}~\bibnamefont {Mischaikow}}, \ and\ \bibinfo {author}
  {\bibfnamefont {M.}~\bibnamefont {Paul}},\ }\bibfield  {title} {\enquote
  {\bibinfo {title} {Correlations between the leading lyapunov vector and
  pattern defects for chaotic rayleigh-b{\'e}nard convection},}\ }\href
  {\doibase https://doi.org/10.1063/1.5071468} {\bibfield  {journal} {\bibinfo
  {journal} {Chaos: An Interdisciplinary Journal of Nonlinear Science}\
  }\textbf {\bibinfo {volume} {29}},\ \bibinfo {pages} {053103} (\bibinfo
  {year} {2019})}\BibitemShut {NoStop}%
\bibitem [{\citenamefont {Wu}\ and\ \citenamefont {Hargreaves}(2021)}]{Wu2021}%
  \BibitemOpen
  \bibfield  {author} {\bibinfo {author} {\bibfnamefont {C.}~\bibnamefont
  {Wu}}\ and\ \bibinfo {author} {\bibfnamefont {C.~A.}\ \bibnamefont
  {Hargreaves}},\ }\bibfield  {title} {\enquote {\bibinfo {title} {Topological
  machine learning for multivariate time series},}\ }\href {\doibase
  10.1080/0952813x.2021.1871971} {\bibfield  {journal} {\bibinfo  {journal}
  {Journal of Experimental \& Theoretical Artificial Intelligence}\ }\textbf
  {\bibinfo {volume} {34}},\ \bibinfo {pages} {311--326} (\bibinfo {year}
  {2021})}\BibitemShut {NoStop}%
\bibitem [{\citenamefont {Kim}\ and\ \citenamefont
  {M{\'e}moli}(2021)}]{multiparameterrank}%
  \BibitemOpen
  \bibfield  {author} {\bibinfo {author} {\bibfnamefont {W.}~\bibnamefont
  {Kim}}\ and\ \bibinfo {author} {\bibfnamefont {F.}~\bibnamefont
  {M{\'e}moli}},\ }\bibfield  {title} {\enquote {\bibinfo {title}
  {Spatiotemporal persistent homology for dynamic metric spaces},}\ }\href
  {\doibase https://doi.org/10.1007/s00454-019-00168-w} {\bibfield  {journal}
  {\bibinfo  {journal} {Discrete \& Computational Geometry}\ }\textbf {\bibinfo
  {volume} {66}},\ \bibinfo {pages} {831--875} (\bibinfo {year}
  {2021})}\BibitemShut {NoStop}%
\bibitem [{\citenamefont {Cohen-Steiner}, \citenamefont {Edelsbrunner},\ and\
  \citenamefont {Morozov}(2006)}]{CohenSteiner2006}%
  \BibitemOpen
  \bibfield  {author} {\bibinfo {author} {\bibfnamefont {D.}~\bibnamefont
  {Cohen-Steiner}}, \bibinfo {author} {\bibfnamefont {H.}~\bibnamefont
  {Edelsbrunner}}, \ and\ \bibinfo {author} {\bibfnamefont {D.}~\bibnamefont
  {Morozov}},\ }\bibfield  {title} {\enquote {\bibinfo {title} {Vines and
  vineyards by updating persistence in linear time},}\ }in\ \href {\doibase
  https://doi.org/10.1145/1137856.1137877} {\emph {\bibinfo {booktitle}
  {Proceedings of the twenty-second annual symposium on Computational
  geometry}}}\ (\bibinfo {year} {2006})\ pp.\ \bibinfo {pages}
  {119--126}\BibitemShut {NoStop}%
\bibitem [{\citenamefont {Ulmer}, \citenamefont {Ziegelmeier},\ and\
  \citenamefont {Topaz}(2019)}]{crocker-Lori2019}%
  \BibitemOpen
  \bibfield  {author} {\bibinfo {author} {\bibfnamefont {M.}~\bibnamefont
  {Ulmer}}, \bibinfo {author} {\bibfnamefont {L.}~\bibnamefont {Ziegelmeier}},
  \ and\ \bibinfo {author} {\bibfnamefont {C.~M.}\ \bibnamefont {Topaz}},\
  }\bibfield  {title} {\enquote {\bibinfo {title} {A topological approach to
  selecting models of biological experiments},}\ }\href
  {https://doi.org/10.1371/journal.pone.0213679} {\bibfield  {journal}
  {\bibinfo  {journal} {PLOS ONE}\ }\textbf {\bibinfo {volume} {14}},\ \bibinfo
  {pages} {1--18} (\bibinfo {year} {2019})}\BibitemShut {NoStop}%
\bibitem [{\citenamefont {Bhaskar}\ \emph {et~al.}(2019)\citenamefont
  {Bhaskar}, \citenamefont {Manhart}, \citenamefont {Milzman}, \citenamefont
  {Nardini}, \citenamefont {Storey}, \citenamefont {Topaz},\ and\ \citenamefont
  {Ziegelmeier}}]{bhaskar2019analyzing}%
  \BibitemOpen
  \bibfield  {author} {\bibinfo {author} {\bibfnamefont {D.}~\bibnamefont
  {Bhaskar}}, \bibinfo {author} {\bibfnamefont {A.}~\bibnamefont {Manhart}},
  \bibinfo {author} {\bibfnamefont {J.}~\bibnamefont {Milzman}}, \bibinfo
  {author} {\bibfnamefont {J.~T.}\ \bibnamefont {Nardini}}, \bibinfo {author}
  {\bibfnamefont {K.~M.}\ \bibnamefont {Storey}}, \bibinfo {author}
  {\bibfnamefont {C.~M.}\ \bibnamefont {Topaz}}, \ and\ \bibinfo {author}
  {\bibfnamefont {L.}~\bibnamefont {Ziegelmeier}},\ }\bibfield  {title}
  {\enquote {\bibinfo {title} {Analyzing collective motion with machine
  learning and topology},}\ }\href {\doibase https://doi.org/10.1063/1.5125493}
  {\bibfield  {journal} {\bibinfo  {journal} {Chaos: An Interdisciplinary
  Journal of Nonlinear Science}\ }\textbf {\bibinfo {volume} {29}},\ \bibinfo
  {pages} {123125} (\bibinfo {year} {2019})}\BibitemShut {NoStop}%
\bibitem [{\citenamefont {Xian}\ \emph {et~al.}(2022)\citenamefont {Xian},
  \citenamefont {Adams}, \citenamefont {Topaz},\ and\ \citenamefont
  {Ziegelmeier}}]{crocker-stack}%
  \BibitemOpen
  \bibfield  {author} {\bibinfo {author} {\bibfnamefont {L.}~\bibnamefont
  {Xian}}, \bibinfo {author} {\bibfnamefont {H.}~\bibnamefont {Adams}},
  \bibinfo {author} {\bibfnamefont {C.~M.}\ \bibnamefont {Topaz}}, \ and\
  \bibinfo {author} {\bibfnamefont {L.}~\bibnamefont {Ziegelmeier}},\
  }\bibfield  {title} {\enquote {\bibinfo {title} {Capturing dynamics of
  time-varying data via topology},}\ }\href {\doibase
  http://dx.doi.org/10.3934/fods.2021033} {\bibfield  {journal} {\bibinfo
  {journal} {Foundations of Data Science}\ }\textbf {\bibinfo {volume} {4}},\
  \bibinfo {pages} {1--36} (\bibinfo {year} {2022})}\BibitemShut {NoStop}%
\bibitem [{\citenamefont {Güzel}, \citenamefont {Munch},\ and\ \citenamefont
  {Khasawneh}(2022)}]{Guezel2022}%
  \BibitemOpen
  \bibfield  {author} {\bibinfo {author} {\bibfnamefont {I.}~\bibnamefont
  {Güzel}}, \bibinfo {author} {\bibfnamefont {E.}~\bibnamefont {Munch}}, \
  and\ \bibinfo {author} {\bibfnamefont {F.}~\bibnamefont {Khasawneh}},\
  }\bibfield  {title} {\enquote {\bibinfo {title} {A case study on identifying
  bifurcation and chaos with {CROCKER} plots},}\ }in\ \href {\doibase
  10.48550/arXiv.2204.06321} {\emph {\bibinfo {booktitle} {Proceedings of TDA
  at SDM (SIAM Data Mining)}}},\ \bibinfo {editor} {edited by\ \bibinfo
  {editor} {\bibfnamefont {R.~W.~R.}\ \bibnamefont {Darling}}, \bibinfo
  {editor} {\bibfnamefont {J.~A.}\ \bibnamefont {Emanuello}}, \bibinfo {editor}
  {\bibfnamefont {E.}~\bibnamefont {Purvine}}, \ and\ \bibinfo {editor}
  {\bibfnamefont {A.}~\bibnamefont {Ridley}}},\ \bibinfo {organization} {SIAM
  Data Mining}\ (\bibinfo  {publisher} {arXiv Proceedings},\ \bibinfo {year}
  {2022})\ \Eprint {http://arxiv.org/abs/2204.06321} {2204.06321} \BibitemShut
  {NoStop}%
\bibitem [{\citenamefont {Kantz}\ and\ \citenamefont
  {Schreiber}(2003)}]{Kantz2004}%
  \BibitemOpen
  \bibfield  {author} {\bibinfo {author} {\bibfnamefont {H.}~\bibnamefont
  {Kantz}}\ and\ \bibinfo {author} {\bibfnamefont {T.}~\bibnamefont
  {Schreiber}},\ }\href {\doibase 10.1017/CBO9780511755798} {\emph {\bibinfo
  {title} {Nonlinear Time Series Analysis}}},\ \bibinfo {edition} {2nd}\ ed.\
  (\bibinfo  {publisher} {Cambridge University Press},\ \bibinfo {year}
  {2003})\BibitemShut {NoStop}%
\bibitem [{\citenamefont {Hirsch}, \citenamefont {Smale},\ and\ \citenamefont
  {Devaney}(2013)}]{Hirsch2003}%
  \BibitemOpen
  \bibfield  {author} {\bibinfo {author} {\bibfnamefont {M.~W.}\ \bibnamefont
  {Hirsch}}, \bibinfo {author} {\bibfnamefont {S.}~\bibnamefont {Smale}}, \
  and\ \bibinfo {author} {\bibfnamefont {R.~L.}\ \bibnamefont {Devaney}},\
  }\href {\doibase 10.1016/B978-0-12-382010-5.00001-4} {\emph {\bibinfo {title}
  {Differential equations, dynamical systems, and an introduction to chaos}}},\
  \bibinfo {edition} {3rd}\ ed.\ (\bibinfo  {publisher} {Elsevier/Academic
  Press, Amsterdam},\ \bibinfo {year} {2013})\ pp.\ \bibinfo {pages}
  {xiv+418}\BibitemShut {NoStop}%
\bibitem [{\citenamefont {Baker}\ and\ \citenamefont
  {Gollub}(1996)}]{baker1996chaotic}%
  \BibitemOpen
  \bibfield  {author} {\bibinfo {author} {\bibfnamefont {G.~L.}\ \bibnamefont
  {Baker}}\ and\ \bibinfo {author} {\bibfnamefont {J.~P.}\ \bibnamefont
  {Gollub}},\ }\href {\doibase https://doi.org/10.1017/CBO9781139170864} {\emph
  {\bibinfo {title} {Chaotic dynamics: an introduction}}}\ (\bibinfo
  {publisher} {Cambridge university press},\ \bibinfo {year}
  {1996})\BibitemShut {NoStop}%
\bibitem [{\citenamefont {Sayama}(2015)}]{sayama2015introduction}%
  \BibitemOpen
  \bibfield  {author} {\bibinfo {author} {\bibfnamefont {H.}~\bibnamefont
  {Sayama}},\ }\href@noop {} {\emph {\bibinfo {title} {Introduction to the
  modeling and analysis of complex systems}}}\ (\bibinfo  {publisher} {Open
  SUNY Textbooks},\ \bibinfo {year} {2015})\BibitemShut {NoStop}%
\bibitem [{\citenamefont {Lorenz}(1963)}]{lorenz1963deterministic}%
  \BibitemOpen
  \bibfield  {author} {\bibinfo {author} {\bibfnamefont {E.~N.}\ \bibnamefont
  {Lorenz}},\ }\bibfield  {title} {\enquote {\bibinfo {title} {Deterministic
  nonperiodic flow},}\ }\href {\doibase
  https://doi.org/10.1175/1520-0469(1963)020\%3C0130:DNF\%3E2.0.CO;2}
  {\bibfield  {journal} {\bibinfo  {journal} {Journal of atmospheric sciences}\
  }\textbf {\bibinfo {volume} {20}},\ \bibinfo {pages} {130--141} (\bibinfo
  {year} {1963})}\BibitemShut {NoStop}%
\bibitem [{\citenamefont {Datseris}(2018)}]{juliaDynamicalSystems}%
  \BibitemOpen
  \bibfield  {author} {\bibinfo {author} {\bibfnamefont {G.}~\bibnamefont
  {Datseris}},\ }\bibfield  {title} {\enquote {\bibinfo {title}
  {Dynamicalsystems.jl: A julia software library for chaos and nonlinear
  dynamics},}\ }\href {\doibase 10.21105/joss.00598} {\bibfield  {journal}
  {\bibinfo  {journal} {Journal of Open Source Software}\ }\textbf {\bibinfo
  {volume} {3}},\ \bibinfo {pages} {598} (\bibinfo {year} {2018})}\BibitemShut
  {NoStop}%
\bibitem [{\citenamefont {Bezanson}\ \emph {et~al.}(2017)\citenamefont
  {Bezanson}, \citenamefont {Edelman}, \citenamefont {Karpinski},\ and\
  \citenamefont {Shah}}]{bezanson2017julia}%
  \BibitemOpen
  \bibfield  {author} {\bibinfo {author} {\bibfnamefont {J.}~\bibnamefont
  {Bezanson}}, \bibinfo {author} {\bibfnamefont {A.}~\bibnamefont {Edelman}},
  \bibinfo {author} {\bibfnamefont {S.}~\bibnamefont {Karpinski}}, \ and\
  \bibinfo {author} {\bibfnamefont {V.~B.}\ \bibnamefont {Shah}},\ }\bibfield
  {title} {\enquote {\bibinfo {title} {Julia: A fresh approach to numerical
  computing},}\ }\href {https://doi.org/10.1137/141000671} {\bibfield
  {journal} {\bibinfo  {journal} {SIAM review}\ }\textbf {\bibinfo {volume}
  {59}},\ \bibinfo {pages} {65--98} (\bibinfo {year} {2017})}\BibitemShut
  {NoStop}%
\bibitem [{\citenamefont {Benettin}\ \emph {et~al.}(1980)\citenamefont
  {Benettin}, \citenamefont {Galgani}, \citenamefont {Giorgilli},\ and\
  \citenamefont {Strelcyn}}]{benettin1980lyapunov}%
  \BibitemOpen
  \bibfield  {author} {\bibinfo {author} {\bibfnamefont {G.}~\bibnamefont
  {Benettin}}, \bibinfo {author} {\bibfnamefont {L.}~\bibnamefont {Galgani}},
  \bibinfo {author} {\bibfnamefont {A.}~\bibnamefont {Giorgilli}}, \ and\
  \bibinfo {author} {\bibfnamefont {J.-M.}\ \bibnamefont {Strelcyn}},\
  }\bibfield  {title} {\enquote {\bibinfo {title} {Lyapunov characteristic
  exponents for smooth dynamical systems and for hamiltonian systems; a method
  for computing all of them. part 1: Theory},}\ }\href
  {https://doi.org/10.1007/BF02128236} {\bibfield  {journal} {\bibinfo
  {journal} {Meccanica}\ }\textbf {\bibinfo {volume} {15}},\ \bibinfo {pages}
  {9--20} (\bibinfo {year} {1980})}\BibitemShut {NoStop}%
\bibitem [{\citenamefont {R\"{o}senstein}, \citenamefont {Collins},\ and\
  \citenamefont {De~Luca}(1993)}]{rosenstein1993practical}%
  \BibitemOpen
  \bibfield  {author} {\bibinfo {author} {\bibfnamefont {M.~T.}\ \bibnamefont
  {R\"{o}senstein}}, \bibinfo {author} {\bibfnamefont {J.~J.}\ \bibnamefont
  {Collins}}, \ and\ \bibinfo {author} {\bibfnamefont {C.~J.}\ \bibnamefont
  {De~Luca}},\ }\bibfield  {title} {\enquote {\bibinfo {title} {A practical
  method for calculating largest lyapunov exponents from small data sets},}\
  }\href {\doibase https://doi.org/10.1016/0167-2789(93)90009-P} {\bibfield
  {journal} {\bibinfo  {journal} {Physica D: Nonlinear Phenomena}\ }\textbf
  {\bibinfo {volume} {65}},\ \bibinfo {pages} {117--134} (\bibinfo {year}
  {1993})}\BibitemShut {NoStop}%
\bibitem [{\citenamefont {Dey}\ and\ \citenamefont {Wang}(2022)}]{Dey2021}%
  \BibitemOpen
  \bibfield  {author} {\bibinfo {author} {\bibfnamefont {T.~K.}\ \bibnamefont
  {Dey}}\ and\ \bibinfo {author} {\bibfnamefont {Y.}~\bibnamefont {Wang}},\
  }\href {\doibase 10.1017/9781009099950} {\emph {\bibinfo {title}
  {Computational Topology for Data Analysis}}}\ (\bibinfo  {publisher}
  {Cambridge University Press},\ \bibinfo {year} {2022})\BibitemShut {NoStop}%
\bibitem [{\citenamefont {Oudot}(2017)}]{Oudot2017}%
  \BibitemOpen
  \bibfield  {author} {\bibinfo {author} {\bibfnamefont {S.~Y.}\ \bibnamefont
  {Oudot}},\ }\href
  {https://www.amazon.com/Persistence-Theory-Representations-Mathematical-Monographs/dp/1470434431?SubscriptionId=0JYN1NVW651KCA56C102&tag=techkie-20&linkCode=xm2&camp=2025&creative=165953&creativeASIN=1470434431}
  {\emph {\bibinfo {title} {Persistence Theory: From Quiver Representations to
  Data Analysis (Mathematical Surveys and Monographs)}}}\ (\bibinfo
  {publisher} {American Mathematical Society},\ \bibinfo {year}
  {2017})\BibitemShut {NoStop}%
\bibitem [{\citenamefont {Munch}(2017)}]{Munch2017}%
  \BibitemOpen
  \bibfield  {author} {\bibinfo {author} {\bibfnamefont {E.}~\bibnamefont
  {Munch}},\ }\bibfield  {title} {\enquote {\bibinfo {title} {A user's guide to
  topological data analysis},}\ }\href {https://doi.org/10.18608/jla.2017.42.6}
  {\bibfield  {journal} {\bibinfo  {journal} {Journal of Learning Analytics}\
  }\textbf {\bibinfo {volume} {4}} (\bibinfo {year} {2017})}\BibitemShut
  {NoStop}%
\bibitem [{\citenamefont {Hatcher}(2002)}]{Hatcher}%
  \BibitemOpen
  \bibfield  {author} {\bibinfo {author} {\bibfnamefont {A.}~\bibnamefont
  {Hatcher}},\ }\href@noop {} {\emph {\bibinfo {title} {Algebraic Topology}}}\
  (\bibinfo  {publisher} {Cambridge University Press},\ \bibinfo {year}
  {2002})\BibitemShut {NoStop}%
\bibitem [{\citenamefont {Munkres}(1984)}]{Munkres2}%
  \BibitemOpen
  \bibfield  {author} {\bibinfo {author} {\bibfnamefont {J.~R.}\ \bibnamefont
  {Munkres}},\ }\href {\doibase https://doi.org/10.1201/9780429493911} {\emph
  {\bibinfo {title} {Elements of Algebraic Topology}}}\ (\bibinfo  {publisher}
  {Addison Wesley},\ \bibinfo {year} {1984})\BibitemShut {NoStop}%
\bibitem [{\citenamefont {Johnson}\ and\ \citenamefont
  {Jung}(2021)}]{johnson2021instability}%
  \BibitemOpen
  \bibfield  {author} {\bibinfo {author} {\bibfnamefont {M.}~\bibnamefont
  {Johnson}}\ and\ \bibinfo {author} {\bibfnamefont {J.-H.}\ \bibnamefont
  {Jung}},\ }\bibfield  {title} {\enquote {\bibinfo {title} {Instability of the
  betti sequence for persistent homology and a stabilized version of the betti
  sequence},}\ }\href {\doibase https://doi.org/10.12941/jksiam.2021.25.296}
  {\bibfield  {journal} {\bibinfo  {journal} {Journal of The Korean Society For
  Industrial and Applied Mathematics}\ }\textbf {\bibinfo {volume} {25}},\
  \bibinfo {pages} {296--311} (\bibinfo {year} {2021})}\BibitemShut {NoStop}%
\bibitem [{\citenamefont {Chung}\ and\ \citenamefont
  {Lawson}(2022)}]{chung2022persistence}%
  \BibitemOpen
  \bibfield  {author} {\bibinfo {author} {\bibfnamefont {Y.-M.}\ \bibnamefont
  {Chung}}\ and\ \bibinfo {author} {\bibfnamefont {A.}~\bibnamefont {Lawson}},\
  }\bibfield  {title} {\enquote {\bibinfo {title} {Persistence curves: A
  canonical framework for summarizing persistence diagrams},}\ }\href {\doibase
  https://doi.org/10.1007/s10444-021-09893-4} {\bibfield  {journal} {\bibinfo
  {journal} {Advances in Computational Mathematics}\ }\textbf {\bibinfo
  {volume} {48}},\ \bibinfo {pages} {1--42} (\bibinfo {year}
  {2022})}\BibitemShut {NoStop}%
\bibitem [{\citenamefont {Cohen-Steiner}\ \emph {et~al.}(2010)\citenamefont
  {Cohen-Steiner}, \citenamefont {Edelsbrunner}, \citenamefont {Harer},\ and\
  \citenamefont {Mileyko}}]{Cohen-Steiner2010}%
  \BibitemOpen
  \bibfield  {author} {\bibinfo {author} {\bibfnamefont {D.}~\bibnamefont
  {Cohen-Steiner}}, \bibinfo {author} {\bibfnamefont {H.}~\bibnamefont
  {Edelsbrunner}}, \bibinfo {author} {\bibfnamefont {J.}~\bibnamefont {Harer}},
  \ and\ \bibinfo {author} {\bibfnamefont {Y.}~\bibnamefont {Mileyko}},\
  }\bibfield  {title} {\enquote {\bibinfo {title} {Lipschitz functions have
  $l_p$-stable persistence},}\ }\href {\doibase 10.1007/s10208-010-9060-6}
  {\bibfield  {journal} {\bibinfo  {journal} {Found. Comput. Math.}\ }\textbf
  {\bibinfo {volume} {10}},\ \bibinfo {pages} {127--139} (\bibinfo {year}
  {2010})}\BibitemShut {NoStop}%
\bibitem [{\citenamefont {Skraba}\ and\ \citenamefont
  {Turner}(2020)}]{Skraba2020}%
  \BibitemOpen
  \bibfield  {author} {\bibinfo {author} {\bibfnamefont {P.}~\bibnamefont
  {Skraba}}\ and\ \bibinfo {author} {\bibfnamefont {K.}~\bibnamefont
  {Turner}},\ }\bibfield  {title} {\enquote {\bibinfo {title} {Wasserstein
  stability for persistence diagrams},}\ }\href@noop {} {\  (\bibinfo {year}
  {2020})},\ \Eprint {http://arxiv.org/abs/2006.16824} {arXiv:2006.16824
  [math.AT]} \BibitemShut {NoStop}%
\bibitem [{\citenamefont {Carlsson}\ and\ \citenamefont
  {Zomorodian}(2009)}]{Carlsson2009}%
  \BibitemOpen
  \bibfield  {author} {\bibinfo {author} {\bibfnamefont {G.}~\bibnamefont
  {Carlsson}}\ and\ \bibinfo {author} {\bibfnamefont {A.}~\bibnamefont
  {Zomorodian}},\ }\bibfield  {title} {{\selectlanguage {English}\enquote
  {\bibinfo {title} {The theory of multidimensional persistence},}\ }}\href
  {\doibase 10.1007/s00454-009-9176-0} {\bibfield  {journal} {\bibinfo
  {journal} {Discrete \& Computational Geometry}\ }\textbf {\bibinfo {volume}
  {42}},\ \bibinfo {pages} {71--93} (\bibinfo {year} {2009})}\BibitemShut
  {NoStop}%
\bibitem [{Note1()}]{Note1}%
  \BibitemOpen
  \bibinfo {note} {The naming convention comes from the association to Betti
  number information, but might \protect \href
  {https://en.wikipedia.org/wiki/Betty_Crocker}{only be obvious to those who
  regularly shop in US grocery stores}.}\BibitemShut {Stop}%
\bibitem [{\citenamefont {Cavanna}, \citenamefont {Jahanseir},\ and\
  \citenamefont {Sheehy}(2015)}]{greedyalgo}%
  \BibitemOpen
  \bibfield  {author} {\bibinfo {author} {\bibfnamefont {N.~J.}\ \bibnamefont
  {Cavanna}}, \bibinfo {author} {\bibfnamefont {M.}~\bibnamefont {Jahanseir}},
  \ and\ \bibinfo {author} {\bibfnamefont {D.~R.}\ \bibnamefont {Sheehy}},\
  }\bibfield  {title} {\enquote {\bibinfo {title} {A geometric perspective on
  sparse filtrations},}\ }\href {https://doi.org/10.48550/arXiv.1506.03797}
  {\bibfield  {journal} {\bibinfo  {journal} {arXiv preprint arXiv:1506.03797}\
  } (\bibinfo {year} {2015})}\BibitemShut {NoStop}%
\bibitem [{\citenamefont {Tralie}, \citenamefont {Saul},\ and\ \citenamefont
  {Bar-On}(2018)}]{ripser}%
  \BibitemOpen
  \bibfield  {author} {\bibinfo {author} {\bibfnamefont {C.}~\bibnamefont
  {Tralie}}, \bibinfo {author} {\bibfnamefont {N.}~\bibnamefont {Saul}}, \ and\
  \bibinfo {author} {\bibfnamefont {R.}~\bibnamefont {Bar-On}},\ }\bibfield
  {title} {\enquote {\bibinfo {title} {{Ripser.py}: A lean persistent homology
  library for python},}\ }\href {\doibase 10.21105/joss.00925} {\bibfield
  {journal} {\bibinfo  {journal} {The Journal of Open Source Software}\
  }\textbf {\bibinfo {volume} {3}},\ \bibinfo {pages} {925} (\bibinfo {year}
  {2018})}\BibitemShut {NoStop}%
\bibitem [{Note2()}]{Note2}%
  \BibitemOpen
  \bibinfo {note} {Every 0-dimensional diagram for a point cloud has an
  infinite bar representing the first connected component which is born and is
  considered to live for every value of $\varepsilon $. Thus the 0-dimensional
  CROCKER plots always have a region of 1 at their maximum $\varepsilon $
  boundary.}\BibitemShut {Stop}%
\bibitem [{\citenamefont {Myers}\ \emph {et~al.}(2020)\citenamefont {Myers},
  \citenamefont {Yesilli}, \citenamefont {Tymochko}, \citenamefont
  {Khasawneh},\ and\ \citenamefont {Munch}}]{teaspoon}%
  \BibitemOpen
  \bibfield  {author} {\bibinfo {author} {\bibfnamefont {A.~D.}\ \bibnamefont
  {Myers}}, \bibinfo {author} {\bibfnamefont {M.}~\bibnamefont {Yesilli}},
  \bibinfo {author} {\bibfnamefont {S.}~\bibnamefont {Tymochko}}, \bibinfo
  {author} {\bibfnamefont {F.}~\bibnamefont {Khasawneh}}, \ and\ \bibinfo
  {author} {\bibfnamefont {E.}~\bibnamefont {Munch}},\ }\bibfield  {title}
  {\enquote {\bibinfo {title} {Teaspoon: A comprehensive python package for
  topological signal processing},}\ }in\ \href@noop {} {\emph {\bibinfo
  {booktitle} {NeurIPS 2020 Workshop on Topological Data Analysis and
  Beyond}}}\ (\bibinfo {year} {2020})\BibitemShut {NoStop}%
\bibitem [{\citenamefont {Parlitz}(2016)}]{Parlitz2016}%
  \BibitemOpen
  \bibfield  {author} {\bibinfo {author} {\bibfnamefont {U.}~\bibnamefont
  {Parlitz}},\ }\enquote {\bibinfo {title} {Estimating lyapunov exponents from
  time series},}\ in\ \href {\doibase 10.1007/978-3-662-48410-4_1} {\emph
  {\bibinfo {booktitle} {Chaos Detection and Predictability}}},\ \bibinfo
  {editor} {edited by\ \bibinfo {editor} {\bibfnamefont {C.~H.}\ \bibnamefont
  {Skokos}}, \bibinfo {editor} {\bibfnamefont {G.~A.}\ \bibnamefont
  {Gottwald}}, \ and\ \bibinfo {editor} {\bibfnamefont {J.}~\bibnamefont
  {Laskar}}}\ (\bibinfo  {publisher} {Springer Berlin Heidelberg},\ \bibinfo
  {address} {Berlin, Heidelberg},\ \bibinfo {year} {2016})\ pp.\ \bibinfo
  {pages} {1--34}\BibitemShut {NoStop}%
\end{thebibliography}%

\end{document}